\documentclass[11pt]{amsart}
\usepackage{amsmath, amssymb, latexsym, graphicx, mathrsfs, enumerate}
\usepackage{amssymb, latexsym, graphicx, mathrsfs, enumerate, amsmath, amsthm}
\usepackage{color}
\input xy
\xyoption{all}

\newlength{\margins}
\usepackage[top=\margins,bottom=\margins,left=\margins,right=\margins]{geometry}

\begin{document}

\newcommand{\alp}{\alpha}
\newcommand{\bet}{\beta}
\newcommand{\gam}{\gamma}
\newcommand{\del}{\delta}
\newcommand{\eps}{\epsilon}
\newcommand{\zet}{\zeta}
\newcommand{\tht}{\theta}
\newcommand{\iot}{\iota}
\newcommand{\kap}{\kappa}
\newcommand{\lam}{\lambda}
\newcommand{\sig}{\sigma}
\newcommand{\ups}{\upsilon}
\newcommand{\ome}{\omega}
\newcommand{\vep}{\varepsilon}
\newcommand{\vth}{\vartheta}
\newcommand{\vpi}{\varpi}
\newcommand{\vrh}{\varrho}
\newcommand{\vsi}{\varsigma}
\newcommand{\vph}{\varphi}
\newcommand{\Gam}{\Gamma}
\newcommand{\Del}{\Delta}
\newcommand{\Tht}{\Theta}
\newcommand{\Lam}{\Lambda}
\newcommand{\Sig}{\Sigma}
\newcommand{\Ups}{\Upsilon}
\newcommand{\Ome}{\Omega}


\newcommand{\frka}{{\mathfrak a}}    \newcommand{\frkA}{{\mathfrak A}}
\newcommand{\frkb}{{\mathfrak b}}    \newcommand{\frkB}{{\mathfrak B}}
\newcommand{\frkc}{{\mathfrak c}}    \newcommand{\frkC}{{\mathfrak C}}
\newcommand{\frkd}{{\mathfrak d}}    \newcommand{\frkD}{{\mathfrak D}}
\newcommand{\frke}{{\mathfrak e}}    \newcommand{\frkE}{{\mathfrak E}}
\newcommand{\frkf}{{\mathfrak f}}    \newcommand{\frkF}{{\mathfrak F}}
\newcommand{\frkg}{{\mathfrak g}}    \newcommand{\frkG}{{\mathfrak G}}
\newcommand{\frkh}{{\mathfrak h}}    \newcommand{\frkH}{{\mathfrak H}}
\newcommand{\frki}{{\mathfrak i}}    \newcommand{\frkI}{{\mathfrak I}}
\newcommand{\frkj}{{\mathfrak j}}    \newcommand{\frkJ}{{\mathfrak J}}
\newcommand{\frkk}{{\mathfrak k}}    \newcommand{\frkK}{{\mathfrak K}}
\newcommand{\frkl}{{\mathfrak l}}    \newcommand{\frkL}{{\mathfrak L}}
\newcommand{\frkm}{{\mathfrak m}}    \newcommand{\frkM}{{\mathfrak M}}
\newcommand{\frkn}{{\mathfrak n}}    \newcommand{\frkN}{{\mathfrak N}}
\newcommand{\frko}{{\mathfrak o}}    \newcommand{\frkO}{{\mathfrak O}}
\newcommand{\frkp}{{\mathfrak p}}    \newcommand{\frkP}{{\mathfrak P}}
\newcommand{\frkq}{{\mathfrak q}}    \newcommand{\frkQ}{{\mathfrak Q}}
\newcommand{\frkr}{{\mathfrak r}}    \newcommand{\frkR}{{\mathfrak R}}
\newcommand{\frks}{{\mathfrak s}}    \newcommand{\frkS}{{\mathfrak S}}
\newcommand{\frkt}{{\mathfrak t}}    \newcommand{\frkT}{{\mathfrak T}}
\newcommand{\frku}{{\mathfrak u}}    \newcommand{\frkU}{{\mathfrak U}}
\newcommand{\frkv}{{\mathfrak v}}    \newcommand{\frkV}{{\mathfrak V}}
\newcommand{\frkw}{{\mathfrak w}}    \newcommand{\frkW}{{\mathfrak W}}
\newcommand{\frkx}{{\mathfrak x}}    \newcommand{\frkX}{{\mathfrak X}}
\newcommand{\frky}{{\mathfrak y}}    \newcommand{\frkY}{{\mathfrak Y}}
\newcommand{\frkz}{{\mathfrak z}}    \newcommand{\frkZ}{{\mathfrak Z}}


\newcommand{\bfa}{{\mathbf a}}    \newcommand{\bfA}{{\mathbf A}}
\newcommand{\bfb}{{\mathbf b}}    \newcommand{\bfB}{{\mathbf B}}
\newcommand{\bfc}{{\mathbf c}}    \newcommand{\bfC}{{\mathbf C}}
\newcommand{\bfd}{{\mathbf d}}    \newcommand{\bfD}{{\mathbf D}}
\newcommand{\bfe}{{\mathbf e}}    \newcommand{\bfE}{{\mathbf E}}
\newcommand{\bff}{{\mathbf f}}    \newcommand{\bfF}{{\mathbf F}}
\newcommand{\bfg}{{\mathbf g}}    \newcommand{\bfG}{{\mathbf G}}
\newcommand{\bfh}{{\mathbf h}}    \newcommand{\bfH}{{\mathbf H}}
\newcommand{\bfi}{{\mathbf i}}    \newcommand{\bfI}{{\mathbf I}}
\newcommand{\bfj}{{\mathbf j}}    \newcommand{\bfJ}{{\mathbf J}}
\newcommand{\bfk}{{\mathbf k}}    \newcommand{\bfK}{{\mathbf K}}
\newcommand{\bfl}{{\mathbf l}}    \newcommand{\bfL}{{\mathbf L}}
\newcommand{\bfm}{{\mathbf m}}    \newcommand{\bfM}{{\mathbf M}}
\newcommand{\bfn}{{\mathbf n}}    \newcommand{\bfN}{{\mathbf N}}
\newcommand{\bfo}{{\mathbf o}}    \newcommand{\bfO}{{\mathbf O}}
\newcommand{\bfp}{{\mathbf p}}    \newcommand{\bfP}{{\mathbf P}}
\newcommand{\bfq}{{\mathbf q}}    \newcommand{\bfQ}{{\mathbf Q}}
\newcommand{\bfr}{{\mathbf r}}    \newcommand{\bfR}{{\mathbf R}}
\newcommand{\bfs}{{\mathbf s}}    \newcommand{\bfS}{{\mathbf S}}
\newcommand{\bft}{{\mathbf t}}    \newcommand{\bfT}{{\mathbf T}}
\newcommand{\bfu}{{\mathbf u}}    \newcommand{\bfU}{{\mathbf U}}
\newcommand{\bfv}{{\mathbf v}}    \newcommand{\bfV}{{\mathbf V}}
\newcommand{\bfw}{{\mathbf w}}    \newcommand{\bfW}{{\mathbf W}}
\newcommand{\bfx}{{\mathbf x}}    \newcommand{\bfX}{{\mathbf X}}
\newcommand{\bfy}{{\mathbf y}}    \newcommand{\bfY}{{\mathbf Y}}
\newcommand{\bfz}{{\mathbf z}}    \newcommand{\bfZ}{{\mathbf Z}}


\newcommand{\cala}{{\mathcal A}}
\newcommand{\calb}{{\mathcal B}}
\newcommand{\calc}{{\mathcal C}}
\newcommand{\cald}{{\mathcal D}}
\newcommand{\cale}{{\mathcal E}}
\newcommand{\calf}{{\mathcal F}}
\newcommand{\calg}{{\mathcal G}}
\newcommand{\calh}{{\mathcal H}}
\newcommand{\cali}{{\mathcal I}}
\newcommand{\calj}{{\mathcal J}}
\newcommand{\calk}{{\mathcal K}}
\newcommand{\call}{{\mathcal L}}
\newcommand{\calm}{{\mathcal M}}
\newcommand{\caln}{{\mathcal N}}
\newcommand{\calo}{{\mathcal O}}
\newcommand{\calp}{{\mathcal P}}
\newcommand{\calq}{{\mathcal Q}}
\newcommand{\calr}{{\mathcal R}}
\newcommand{\cals}{{\mathcal S}}
\newcommand{\calt}{{\mathcal T}}
\newcommand{\calu}{{\mathcal U}}
\newcommand{\calv}{{\mathcal V}}
\newcommand{\calw}{{\mathcal W}}
\newcommand{\calx}{{\mathcal X}}
\newcommand{\caly}{{\mathcal Y}}
\newcommand{\calz}{{\mathcal Z}}


\newcommand{\scra}{{\mathscr A}}
\newcommand{\scrb}{{\mathscr B}}
\newcommand{\scrc}{{\mathscr C}}
\newcommand{\scrd}{{\mathscr D}}
\newcommand{\scre}{{\mathscr E}}
\newcommand{\scrf}{{\mathscr F}}
\newcommand{\scrg}{{\mathscr G}}
\newcommand{\scrh}{{\mathscr H}}
\newcommand{\scri}{{\mathscr I}}
\newcommand{\scrj}{{\mathscr J}}
\newcommand{\scrk}{{\mathscr K}}
\newcommand{\scrl}{{\mathscr L}}
\newcommand{\scrm}{{\mathscr M}}
\newcommand{\scrn}{{\mathscr N}}
\newcommand{\scro}{{\mathscr O}}
\newcommand{\scrp}{{\mathscr P}}
\newcommand{\scrq}{{\mathscr Q}}
\newcommand{\scrr}{{\mathscr R}}
\newcommand{\scrs}{{\mathscr S}}
\newcommand{\scrt}{{\mathscr T}}
\newcommand{\scru}{{\mathscr U}}
\newcommand{\scrv}{{\mathscr V}}
\newcommand{\scrw}{{\mathscr W}}
\newcommand{\scrx}{{\mathscr X}}
\newcommand{\scry}{{\mathscr Y}}
\newcommand{\scrz}{{\mathscr Z}}


\newcommand{\AAA}{{\mathbb A}} 
\newcommand{\BB}{{\mathbb B}}
\newcommand{\CC}{{\mathbb C}}
\newcommand{\DD}{{\mathbb D}}
\newcommand{\EE}{{\mathbb E}}
\newcommand{\FF}{{\mathbb F}}
\newcommand{\GG}{{\mathbb G}}
\newcommand{\HH}{{\mathbb H}}
\newcommand{\II}{{\mathbb I}}
\newcommand{\JJ}{{\mathbb J}}
\newcommand{\KK}{{\mathbb K}}
\newcommand{\LL}{{\mathbb L}}
\newcommand{\MM}{{\mathbb M}}
\newcommand{\NN}{{\mathbb N}}
\newcommand{\OO}{{\mathbb O}}
\newcommand{\PP}{{\mathbb P}}
\newcommand{\QQ}{{\mathbb Q}}
\newcommand{\RR}{{\mathbb R}}
\newcommand{\SSS}{{\mathbb S}} 
\newcommand{\TT}{{\mathbb T}}
\newcommand{\UU}{{\mathbb U}}
\newcommand{\VV}{{\mathbb V}}
\newcommand{\WW}{{\mathbb W}}
\newcommand{\XX}{{\mathbb X}}
\newcommand{\YY}{{\mathbb Y}}
\newcommand{\ZZ}{{\mathbb Z}}


\newcommand{\tta}{\hbox{\tt a}}    \newcommand{\ttA}{\hbox{\tt A}}
\newcommand{\ttb}{\hbox{\tt b}}    \newcommand{\ttB}{\hbox{\tt B}}
\newcommand{\ttc}{\hbox{\tt c}}    \newcommand{\ttC}{\hbox{\tt C}}
\newcommand{\ttd}{\hbox{\tt d}}    \newcommand{\ttD}{\hbox{\tt D}}
\newcommand{\tte}{\hbox{\tt e}}    \newcommand{\ttE}{\hbox{\tt E}}
\newcommand{\ttf}{\hbox{\tt f}}    \newcommand{\ttF}{\hbox{\tt F}}
\newcommand{\ttg}{\hbox{\tt g}}    \newcommand{\ttG}{\hbox{\tt G}}
\newcommand{\tth}{\hbox{\tt h}}    \newcommand{\ttH}{\hbox{\tt H}}
\newcommand{\tti}{\hbox{\tt i}}    \newcommand{\ttI}{\hbox{\tt I}}
\newcommand{\ttj}{\hbox{\tt j}}    \newcommand{\ttJ}{\hbox{\tt J}}
\newcommand{\ttk}{\hbox{\tt k}}    \newcommand{\ttK}{\hbox{\tt K}}
\newcommand{\ttl}{\hbox{\tt l}}    \newcommand{\ttL}{\hbox{\tt L}}
\newcommand{\ttm}{\hbox{\tt m}}    \newcommand{\ttM}{\hbox{\tt M}}
\newcommand{\ttn}{\hbox{\tt n}}    \newcommand{\ttN}{\hbox{\tt N}}
\newcommand{\tto}{\hbox{\tt o}}    \newcommand{\ttO}{\hbox{\tt O}}
\newcommand{\ttp}{\hbox{\tt p}}    \newcommand{\ttP}{\hbox{\tt P}}
\newcommand{\ttq}{\hbox{\tt q}}    \newcommand{\ttQ}{\hbox{\tt Q}}
\newcommand{\ttr}{\hbox{\tt r}}    \newcommand{\ttR}{\hbox{\tt R}}
\newcommand{\tts}{\hbox{\tt s}}    \newcommand{\ttS}{\hbox{\tt S}}
\newcommand{\ttt}{\hbox{\tt t}}    \newcommand{\ttT}{\hbox{\tt T}}
\newcommand{\ttu}{\hbox{\tt u}}    \newcommand{\ttU}{\hbox{\tt U}}
\newcommand{\ttv}{\hbox{\tt v}}    \newcommand{\ttV}{\hbox{\tt V}}
\newcommand{\ttw}{\hbox{\tt w}}    \newcommand{\ttW}{\hbox{\tt W}}
\newcommand{\ttx}{\hbox{\tt x}}    \newcommand{\ttX}{\hbox{\tt X}}
\newcommand{\tty}{\hbox{\tt y}}    \newcommand{\ttY}{\hbox{\tt Y}}
\newcommand{\ttz}{\hbox{\tt z}}    \newcommand{\ttZ}{\hbox{\tt Z}}

\newcommand{\phm}{\phantom}
\newcommand{\ds}{\displaystyle }
\newcommand{\smallstrut}{\vphantom{\vrule height 3pt }}
\def\bdm #1#2#3#4{\left(
\begin{array} {c|c}{\ds{#1}}
 & {\ds{#2}} \\ \hline
{\ds{#3}\vphantom{\ds{#3}^1}} &  {\ds{#4}}
\end{array}
\right)}
\newcommand{\wtd}{\widetilde }
\newcommand{\bsl}{\backslash }
\newcommand{\GL}{{\mathrm{GL}}}
\newcommand{\SL}{{\mathrm{SL}}}
\newcommand{\GSp}{{\mathrm{GSp}}}
\newcommand{\PGSp}{{\mathrm{PGSp}}}
\newcommand{\SP}{{\mathrm{Sp}}}
\newcommand{\SO}{{\mathrm{SO}}}
\newcommand{\SU}{{\mathrm{SU}}}
\newcommand{\Ind}{\mathrm{Ind}}
\newcommand{\Hom}{{\mathrm{Hom}}}
\newcommand{\Ad}{{\mathrm{Ad}}}
\newcommand{\Sym}{{\mathrm{Sym}}}
\newcommand{\Mat}{\mathrm{M}}
\newcommand{\sgn}{\mathrm{sgn}}
\newcommand{\trs}{\,^t\!}
\newcommand{\iu}{\sqrt{-1}}
\newcommand{\oo}{\hbox{\bf 0}}
\newcommand{\ono}{\hbox{\bf 1}}
\newcommand{\smallcirc}{\lower .3em \hbox{\rm\char'27}\!}
\newcommand{\bAf}{\bA_{\hbox{\eightrm f}}}
\newcommand{\thalf}{{\textstyle{\frac12}}}
\newcommand{\shp}{\hbox{\rm\char'43}}
\newcommand{\Gal}{\operatorname{Gal}}

\newcommand{\bdel}{{\boldsymbol{\delta}}}
\newcommand{\bchi}{{\boldsymbol{\chi}}}
\newcommand{\bgam}{{\boldsymbol{\gamma}}}
\newcommand{\bome}{{\boldsymbol{\omega}}}
\newcommand{\bpsi}{{\boldsymbol{\psi}}}
\newcommand{\GK}{\mathrm{GK}}
\newcommand{\ord}{\mathrm{ord}}
\newcommand{\diag}{\mathrm{diag}}
\newcommand{\ua}{{\underline{a}}}
\newcommand{\ZZn}{\ZZ_{\geq 0}^n}
\newcommand{\calhnd}{{\mathcal H}^\mathrm{nd}}
\newcommand{\EGK}{\mathrm{EGK}}

\theoremstyle{definition}
\newtheorem{theorem}{Theorem}[section]
\newtheorem{lemma}[theorem]{Lemma}
\newtheorem{proposition}[theorem]{Proposition}
\newtheorem{conjecture}{Conjecture}[section]
\newtheorem{definition}{Definition}[section]
\newtheorem{statement}[theorem]{Statement}
\newtheorem{question}[theorem]{Problem}
\theoremstyle{remark}
\newtheorem{remark}{Remark}[section]
\newtheorem{corollary}[theorem]{Corollary}
\newtheorem{example}{Example}[section]
\numberwithin{equation}{section}

\title[On the local density formula and the Gross-Keating invariant]
{On the local density formula and the Gross-Keating invariant  with an Appendix `The local density of a binary quadratic form' by T. Ikeda and H. Katsurada}

\thanks{The  author is partially supported by JSPS KAKENHI Grant No. 16F16316,  Samsung Science and Technology Foundation under Project Number SSTF-BA1802-03, and NRF-2018R1A4A1023590.}

\author[Sungmun Cho]{Sungmun Cho}
\address{Sungmun Cho \\ 
Department of Mathematics, POSTECH, 77, Cheongam-ro, Nam-gu, Pohang-si, Gyeongsangbuk-do, 37673, KOREA}
\email{sungmuncho12@gmail.com}

\maketitle

\begin{abstract}
T. Ikeda and H. Katsurada have developed the theory of the  Gross-Keating invariant  of a quadratic form in their recent papers \cite{IK1} and \cite{IK2}.
In particular, they prove that the local factors of the Fourier coefficients of the Siegel-Eisenstein series are completely determined by the Gross-Keating invariants with extra datums, called the extended GK datums, in \cite{IK2}.

On the other hand, such a local factor is a special case of the local density for a pair of two quadratic forms.
Thus we propose a  general question  if the local density can be expressed in terms of a certain series of  
 extended GK datums.

In this paper, we prove that the answer to this question is affirmative, for the local density of a single quadratic form defined over an unramified finite extension of $\mathbb{Z}_2$ and over a finite extension of $\mathbb{Z}_p$ with $p$ odd.
In the appendix, T.~Ikeda and H.~Katsurada compute the local density formula of a single binary quadratic form defined over any finite extension of $\mathbb{Z}_2$.
\end{abstract}

\tableofcontents

\section{Introduction}\label{intro}
In 1993, B. Gross and K. Keating defined a certain invariant of a ternary quadratic form over $\mathbb{Z}_p$, in order to formulate an expression for the arithmetic intersection number of three cycles associated to three modular polynomials over the moduli stack of pairs of elliptic curves in \cite{GK}.
This invariant has
been generalized to quadratic forms of any degree over a local field, and is now
called the Gross-Keating invariant.

The Gross-Keating invariant had been almost forgotten\footnote{The Gross-Keating invariant had been treated in \cite{Bouw}. The first sentence of loc. cit. says `This note provides details on \cite{GK} Section 4.'} for a while after the work of Gross and Keating.
It was T. Ikeda and H. Katsurada who recently developed the theory of the Gross-Keating invariant in \cite{IK1} and discovered its importance to the study of the Fourier coefficients of the Siegel-Eisenstein series\footnote{We follow the notion and the definition of \cite{Katsurada} for the Siegel-Eisenstein series.} with any degree and with any weight (see \cite{IK2}).
Furthermore, it has been revealed that the Gross-Keating invariant plays a key role in investigating an analogy between intersection numbers on orthogonal Shimura varieties  and the Fourier coefficients of the Siegel-Eisenstein series in Kudla's program in \cite{CY}.
A formula for the  Gross-Keating invariant over an unramified finite extension of $\mathbb{Z}_2$ is derived in \cite{CIKY1} and \cite{CIKY2}.  \\

On the other hand, the local density, denoted by $\alpha(L, L')$,  of a pair of two quadratic lattices $(L, Q_L)$ and $(L', Q_{L'})$ 
defined over a finite extension of $\mathbb{Z}_p$, provides  vital information towards computing  the number of representations of a global quadratic form, which is a central problem in the theory of Siegel-Weil formula  as well as in the arithmetic theory of quadratic forms.
Special but important cases of the local density $\alpha(L, L')$ are as follows:
\begin{enumerate}
\item if $L'$ is a hyperbolic space defined over $\mathbb{Z}_p$ so that $L'$ is equivalent to 
\[
\begin{pmatrix} 0& 1/2 \\ 1/2 & 0 \end{pmatrix}\perp \cdots \perp \begin{pmatrix} 0& 1/2 \\ 1/2 & 0 \end{pmatrix},
\]
then the associated local density $\alpha(L, L')$ is the local factor at $p$ of the Fourier coefficients of the Siegel-Eisenstein series.
For a detailed explanation, see \cite{Katsurada}.

\item If $L=L'$, then the local density $\alpha(L, L)$  is the local factor at $p$ of the Smith-Minkowski-Siegel mass formula, which is an essential tool for the classification of integral quadratic lattices (over a finite extension of $\mathbb{Z}$).
We refer to the introduction of \cite{Cho} for the history of the local density of a single quadratic form.
\end{enumerate}

Recently in \cite{IK2}, Ikeda and Katsurada show that the local density in the above case (1) is completely determined by the Gross-Keating invariant with extra datum, called the extended GK datum. 
Along with their observation, we generalize their philosophy formulated in the following question:

\begin{question}\label{question}
Can we give a rule for associating to each pair $(L, L')$ of quadratic lattices a sequence $\mathcal{L}(L,L')$ of quadratic lattices with the property that if $\mathrm{EGK}(\mathcal{L}(L,L'))=\mathrm{EGK}(\mathcal{L}(M,M'))$ as multisets then 
\[
\textit{$\alpha(L, L')=\alpha(M, M')$}?
\]
\end{question}

The purpose of this paper is to   answer  this question in the case (2) listed above when $L=L'$ is defined over a finite unramified extension of $\mathbb{Z}_2$.
In the author's previous paper \cite{Cho}, the local density formula of this case is given in terms of certain smooth group schemes. 
The main theorem of our paper is the following:

\begin{theorem}\label{maintheorem}(Theorem \ref{thm-ldgk})
For a quadratic lattice $(L, Q_L)$ defined over a finite unramified extension $\mathfrak{o}$ of $\mathbb{Z}_2$,
the local density $\alpha(L,L)$ is completely determined by the collection consisting of  $\mathrm{GK}(L\oplus -L)$ 
together with the $\mathrm{EGK}(L\cap 2^i  L^\sharp)^{\leq 1}$ as $i$ runs over all integers  such that $L_i$ is nonzero, where $L=\oplus_i L_i$ is a Jordan splitting.
In other words,  given any two quadratic lattices $L, M$ satisfying 
\[
\left\{
  \begin{array}{l}
 \textit{$\mathrm{GK}(L\oplus -L)=\mathrm{GK}(M\oplus -M)$};\\
 \textit{$\mathrm{EGK}(L\cap 2^i  L^\sharp)^{\leq 1}=\mathrm{EGK}(M\cap 2^i  M^\sharp)^{\leq 1}$ for all $i$},
    \end{array} \right.
\]
we have that 
\[\alpha(L,L)=\alpha(M,M).\]
Here, 
$\GK(L)$ is defined in Definition \ref{def:2.1}, $\EGK(L)$ is defined in Definition \ref{def:3.3}, and $\mathrm{EGK}(L\cap 2^i  L^\sharp)^{\leq 1}$ is defined Definition \ref{tegk}.
The lattice $L^\sharp$ is the dual lattice of $L$ and
the lattice $L\cap 2^i  L^\sharp$ is  characterized as follows:
\[
L\cap 2^i  L^\sharp=\{x\in L \mid \langle x,L\rangle_{Q_L} \in 2^i\frko\},
\]
where 
$\langle -,-\rangle_{Q_L}$ is the symmetric bilinear form associated to $Q_L$ such that $\langle x,x\rangle_{Q_L}=Q_L(x)$ for $x\in L$.

\end{theorem}

When $p$ is odd,  it is easy to see that the local density $\alpha(L, L)$ is completely determined by $\mathrm{GK}(L)$, as explained in Remark \ref{podd}.

This paper is organized as follows. In Section \ref{sec:1}, we will explain notations and definitions of the Gross-Keating invariant and the extended GK datum,  taken from \cite{IK1} for synchronization.
In Section \ref{sec:5.1}, we will recall the local density formula given in the author's previous paper \cite{Cho}.
In Section \ref{sec:5.2}, we will prove  the above Theorem \ref{maintheorem}, by introducing the `truncated' extended GK datum (cf. Section \ref{sss3}) which is much simpler than the extended GK datum.
An appendix to this paper has been written by T.~Ikeda  and H.~Katsurada to compute the local density of binary quadratic forms over any finite extension of $\mathbb{Z}_2$.

\begin{remark}
In the appendix, Ikeda and Katsurada also compute $\GK(L\perp -L)$ and $\EGK(L\cap \vpi^i L^\sharp)^{\leq 1}$ for a  quadratic lattice $L$ of rank $2$ over an arbitrary finite extension of $\mathbb{Z}_2$.
Here we refer  $\vpi$ to Section \ref{ssnota}.
But they show that the local density is not determined by these  invariants (See Example \ref{ex:A1}).
Thus, towards  Problem \ref{question}, we may need more subtle invariants to determine the local density.
\end{remark}

\textbf{Acknowledgments.} The author would like to express  deep appreciation to Professors T. Ikeda and H. Katsurada for many fruitful discussions  and for providing the appendix.
We would also like to  thank the referee for helpful suggestions and comments which substantially helped with the presentation of our paper.

\section{Notation and definition}
\label{sec:1}
\subsection{Notation}\label{ssnota}
\begin{itemize}
\item Let $F$ be a finite  field extension of $\mathbb{Q}_p$, and $\frko=\frko_F$ its ring of integers.
Let $\mathfrak{o}^{\times}$ be the set of units in $\mathfrak{o}$.
The maximal ideal and  the residue field of $\frko$ are denoted by $\frkp$ and $\frkk$, respectively.
We put $q=[\frko:\frkp]$.

\item $F$ is said to be dyadic if $q$ is even.

\item We fix a prime element $\vpi$ of $\frko$ once and for all.

\item The order of $x\in F^\times$ is given by $\mathrm{ord}(x)=n$ for $x\in \vpi^n \frko^\times$.
We understand $\ord(0)=+\infty$.
The order of $x$ is sometimes referred to as the exponential valuation of $x$.
\item Put $F^{\times 2}=\{x^2\,|\, x\in F^\times\}$.
Similarly, we put $\frko^{\times 2}=\{x^2\,|\, x\in\frko^\times\}$.

\item  We consider an $\mathfrak{o}$-lattice $L$ with a quadratic form $Q_L:L \rightarrow \mathfrak{o}$.
Here, an $\mathfrak{o}$-lattice means a finitely generated free $\mathfrak{o}$-module. 
Such a quadratic form $Q_L$ is called \textit{an integral quadratic form} and 
such a pair  $(L, Q_L)$  is called  \textit{a quadratic lattice}.
We sometimes say that $L$ is a quadratic lattice by omitting $Q_L$, if this does not cause confusion or ambiguity. 
Let $\langle -,-\rangle_{Q_L}$ be the symmetric bilinear form on $L$ such that
$$\langle x,y\rangle_{Q_L}=\frac{1}{2}(Q_L(x+y)-Q_L(x)-Q_L(y)).$$
Note that a bilinear form $\langle -,-\rangle_{Q_L}$ is valued in $\frac{1}{•2}\mathfrak{o}$ and thus corresponds to an element of $\mathcal{H}_n(\mathfrak{o})$, which will be defined in Section \ref{sec:2}.
We assume that 
$V=L\otimes_{\mathfrak{o}} F$ is non-degenerate with respect to $\langle -,-\rangle_{Q_L}$.

\item A quadratic lattice $L$ is the \textit{orthogonal sum} of sublattices $L_1$ and $L_2$, written $L=L_1\oplus L_2$, if $L_1\cap L_2=0$, $L_1$ is orthogonal to $L_2$ with respect to the symmetric bilinear form $\langle-,- \rangle_{Q_L}$, and $L_1$ and $L_2$ together span $L$.


\item When $R$ is a ring, the set of $m\times n$ matrices with entries in $R$ is denoted by $\mathrm{M}_{mn}(R)$ or $\mathrm{M}_{m,n}(R)$.
As usual, $\mathrm{M}_n(R)=\mathrm{M}_{n,n}(R)$.
\item The identity matrix of size $n$ is denoted by $\mathbf{1}_n$.
\item For $X_1\in \mathrm{M}_s(R)$ and $X_2\in\mathrm{M}_t(R)$, the matrix $\begin{pmatrix} X_1 & 0 \\ 0 & X_2\end{pmatrix}\in\mathrm{M}_{s+t}(R)$ is denoted by $X_1\perp X_2$.
\item The diagonal matrix whose diagonal entries are $b_1$, $\ldots$, $b_n$ is denoted by $\mathrm{diag}(b_1, \dots, b_n)=(b_1)\perp\dots\perp (b_n)$.

\item 
Let $(a_1, \cdots, a_m)$ and $(b_1, \cdots, b_n)$ be  non-decreasing sequences consisting of non-negative integers.
Then $(a_1, \cdots, a_m)\cup  (b_1, \cdots, b_n)$ is defined as the non-decreasing sequence  $(c_1, \cdots, c_{n+m})$
such that  $\{c_1, \cdots, c_{n+m}\}=\{a_1, \cdots, a_m\}\cup  \{b_1, \cdots, b_n\}$ as multisets.
For example, $(0,1,4)\cup (1,3)=(0,1,1,3,4)$.


\item For $\underline{a}=(a_1, \cdots, a_n)$ 
with each $a_i$ an element of $\mathbb{Z}$,
the sum $a_1+\cdots +a_n$ is denoted by $|\underline{a}|$.

\item For $\underline{a}=(a_1, \cdots, a_n)$ 
with each $a_i$ an element of $\mathbb{Z}$,
the  $m$-tuple $(a_1, \cdots, a_m)$ with $m\leq n$ is denoted  by $\underline{a}^{(m)}$.

\item The set of symmetric matrices $B\in \mathrm{M}_n(F)$ of size $n$ is denoted by $\mathrm{Sym}_n(F)$.
Similarly, define $\mathrm{Sym}_n(\mathfrak{o})$.
\item For $B\in \mathrm{Sym}_n(F)$ and $X\in\mathrm{M}_{n,m}(F)$, we set $B[X]={}^t\! XBX \in \mathrm{Sym}_m(F)$.
\item When $G$ is a subgroup of $\GL_n(F)$, we shall say that two elements $B_1, B_2\in\mathrm{Sym}_n(F)$ are  $G$-equivalent, if there is an element $X\in G$ such that $B_1[X]=B_2$.
\end{itemize}

\subsection{Gross-Keating invariants}
\label{sec:2}
In this subsection, we explain the definition of  the Gross-Keating invariant
 and collect some theorems, taken from  \cite{IK1}.
 
 \begin{itemize}
\item 
We say that $B=(b_{ij})\in \mathrm{Sym}_n(F)$ is a half-integral symmetric matrix if
\begin{align*}
b_{ii}\in\frko_F &\qquad (1\leq i\leq n),  \\
2b_{ij}\in\frko_F& \qquad (1\leq i\leq j\leq n).
\end{align*}
\item The set of all half-integral symmetric matrices of size $n$ is denoted by $\calh_n(\frko)$.
\item An element $B\in\calh_n(\frko)$ is non-degenerate if $\det B\neq 0$.
\item The set of all non-degenerate elements of $\calh_n(\frko)$ is denoted by $\calhnd_n(\frko)$.
\item For $B=(b_{ij})_{1\leq i,j\leq n}\in\calh_n(\frko)$ and $1\leq m\leq n$,  we denote the upper left $m\times m$ submatrix $(b_{ij})_{1\leq i, j\leq m}\in\calh_m(\frko)$ by $B^{(m)}$.

\item For $B=(b_{ij})\in \calhnd_n(\frko)$, we say that the quadratic lattice $(L, Q_L)$ is represented by $B$ if $L$ is of rank $n$ and there is an ordered basis $(e_1, \cdots, e_n)$ of $L$ such that  $b_{ij}=\langle e_i,e_j\rangle_{Q_L}$.

\item When two elements $B, B'\in\calh_n(\frko)$ are $\GL_n(\frko)$-equivalent, we just say they are equivalent and write $B\sim B'$.
For two quadratic lattices $(L, Q_L)$ and $(L', Q_{L'})$ represented by  $B$ and $B'$ respectively, we say that $(L, Q_L)$ and $(L', Q_{L'})$ are equivalent if $B$ and $B'$ are equivalent.
We sometimes say that $Q_L$ and $Q_{L'}$ are equivalent, if this does not cause confusion or ambiguity.

\item The equivalence class of $B$ is denoted by $\{B\}_{equiv}$, i.e., 
$\{B\}_{equiv}=\{B[U]\,|\, U\in\GL_n(\frko)\}$.
 \end{itemize}
 
\begin{definition}(\cite{IK1}, Definitions 0.1 and 0.2)
\label{def:2.1}
\begin{enumerate}
\item Let $B=(b_{ij})\in\calhnd_n(\frko)$.
Let $S(B)$ be the set of all non-decreasing sequences $(a_1, \ldots, a_n)\in\ZZn$ such that
\begin{align*}
&\ord(b_{ii})\geq a_i \qquad\qquad\qquad\quad (1\leq i\leq n), \\
&\ord(2 b_{ij})\geq (a_i+a_j)/2  \qquad\; (1\leq i\leq j\leq n).
\end{align*}
Put
\[
\bfS(\{B\}_{equiv})=\bigcup_{B'\in\{B\}_{equiv}} S(B')=\bigcup_{U\in\GL_n(\frko)} S(B[U]).
\]
The Gross-Keating invariant $\GK(B)$ of $B$ is the greatest element of $\bfS(\{B\}_{equiv})$ with respect to the lexicographic order $\succeq$ on $\ZZn$.

\item A symmetric matrix $B (\in\calhnd_n(\frko))$ is called $optimal$ if $\mathrm{GK}(B)\in S(B)$.

\item If $B$ is a symmetric matrix of a quadratic lattice $(L, Q_L)$, then  $\mathrm{GK}(L)$, called the Gross-Keating invariant of  $(L, Q_L)$, is defined by $\mathrm{GK}(B)$.
$\mathrm{GK}(L)$ is independent of the choice of the matrix $B$.
\end{enumerate}
\end{definition}

It is known that the set $\mathbf{S}(\{B\}_{equiv})$ is  finite (cf. \cite{IK1}), which explains the well-definedness of $\mathrm{GK}(B)$.
We can also see that $\mathrm{GK}(B)$ depends only on the equivalence class of  $B$.
A sequence of length $0$ is denoted by $\emptyset$.
When $B$ is the empty matrix, we understand $\GK(B)=\emptyset$.
By definition, a non-degenerate half-integral symmetric matrix $B\in\calhnd_n(\frko)$ is equivalent to an optimal form.\\

For $B\in\calhnd_n(\frko)$, we put $D_B=(-4)^{[n/2]}\det B$.
If $n$ is even, we denote the discriminant ideal of $F(\sqrt{D_B})/F$ by $\mathfrak{D} _B$.
We  put
\[
\xi_B=
\begin{cases} 
1 & \text{ if $D_B\in F^{\times 2}$,} \\
-1 & \text{ if $F(\sqrt{D_B})/F$ is unramified and $[F(\sqrt{D_B}):F]=2$,} \\
0 & \text{ if $F(\sqrt{D_B})/F$ is ramified.} 
\end{cases}
\]
\begin{definition}(\cite{IK1}, Definition 0.3)
\label{def:0.3}
For $B\in\calhnd_n(\frko)$, we put
\[
\Del(B)=
\begin{cases}
\ord(D_B) & \text{ if $n$ is odd,} \\
\ord(D_B)-\ord(\mathfrak{D}_B)+1-\xi_B^2 & \text{ if $n$ is even.}
\end{cases}
\]
\end{definition}
Note that if $n$ is even, then
\[
\Del(B)=
\begin{cases}
\ord(D_B) & \text{ if $\ord(\mathfrak{D}_B)=0$,} \\
\ord(D_B)-\ord(\mathfrak{D}_B)+1 & \text{ if $\ord(\mathfrak{D}_B)>0$.}
\end{cases}
\]

One of the main results of \cite{IK1} is the following theorem:
\begin{theorem}[\cite{IK1}, Theorem 0.1] 
\label{thm:2.1}
For $B\in\calhnd_n(\frko)$,  we have
\[
|\GK(B)|=\Del(B).
\]
\end{theorem}


\begin{definition}(\cite{IK1}, Definition 0.4)
\label{def:2.4}
The Clifford invariant of $B\in\calhnd_n(\frko)$ is the Hasse invariant of the Clifford algebra (resp.~the even Clifford algebra) of $B$ if $n$ is even (resp.~odd).
\end{definition}
We denote the Clifford invariant of $B$ by $\eta_B$.
If $B$ is $\GL_n(F)$-equivalent to $\mathrm{diag}(b'_1, \ldots, b'_n)$, then 
\begin{align*}
\eta_B
=&
\langle -1, -1 \rangle^{[(n+1)/4]}\langle -1, \det B \rangle^{[(n-1)/2]} \prod_{i < j} \langle b'_i, b'_j \rangle \\
=&\begin{cases}
\langle -1, -1 \rangle^{m(m-1)/2}\langle -1, \det B \rangle^{m-1} \ds\prod_{i < j} \langle b'_i, b'_j \rangle & \text{ if $n=2m$, } \\
\noalign{\vskip 6pt}
\langle -1, -1 \rangle^{m(m+1)/2}\langle -1, \det B \rangle^{m} \ds\prod_{i < j} \langle b'_i, b'_j \rangle & \text{ if $n=2m+1$. } 
\end{cases}
\end{align*}
Here, $\langle -,-  \rangle$ is the quadratic Hilbert symbol.
If $H\in\calhnd_2(\frko)$ is $\GL_2(F)$-isomorphic to a hyperbolic plane, then $\eta_{B\perp H}=\eta_B$.
In particular, if $n$ is odd, then we have
\begin{align*}
\eta_B
=&
\begin{cases} 
1 & \text{ if  $B$ is split over $F$, that is,  the associated Witt index  is $\frac{n-1}{2}$,} \\
-1 & \text{ otherwise.} 
\end{cases}
\end{align*}

The following theorem is necessary to define the extended GK datum, which will be explained in the next subsection.
\begin{theorem}[\cite{IK1}, Theorem 0.4]  
\label{thm:2.4}
Let $B, B_1\in\calhnd_n(\frko)$.
Suppose that $B\sim B_1$ and both $B$ and $B_1$ are optimal.
Let $\ua=(a_1, a_2, \ldots, a_n)=\GK(B)=\GK(B_1)$.
Suppose that $a_k<a_{k+1}$ for $1\leq k < n$.
Then the following assertions (1) and (2) hold.
\begin{itemize}
\item[(1)] If $k$ is even, then $\xi_{B^{(k)}}=\xi_{B_1^{(k)}}$.
\item[(2)] If $k$ is odd, then $\eta_{B^{(k)}}=\eta_{B_1^{(k)}}$.
\end{itemize}
\end{theorem}

\subsection{The extended GK datum} 
\label{sec:3}

Ikeda and Katsurada augment
the notion of the Gross-Keating invariant with additional data, resulting in an invariant that they call the extended
GK datum, whose definition we now recall in detail from  \cite{IK1}.

\begin{definition}[\cite{IK2}, Definition 3.1]\label{def3.0}
Let $\underline{a}=(a_1, \cdots, a_n)$ be a non-decreasing sequence of non-negative integers.
Write $\underline{a}$ as
\[\underline{a}=(\underbrace{m_1, \cdots, m_1}_{n_1}, \cdots, \underbrace{m_r, \cdots, m_r}_{n_r} )\]
with $m_1<\cdots <m_r$ and $n=n_1+\cdots + n_r$.
For $s=1, 2, \cdots, r$, put
\[n_s^{\ast}=\sum_{u=1}^{s}n_u,\]
and
\[ I_s=\{ n_{s-1}^{\ast}+1, n_{s-1}^{\ast}+2, \cdots, n_{s}^{\ast}\}. \]
Here, we let $n_0^{\ast}=0$.
\end{definition}

\begin{definition} 
\label{def:3.3}(\cite{IK1}, Definition 6.3)
We define the extended GK datum as follows.
\begin{enumerate}
\item
Let  $B \in \calhnd_n(\frko )$ be an optimal form such that $\GK(B)=\ua=(a_1,\ldots,a_n)$.
We define $\zeta_s=\zeta_s(B)$ by 
\[
\zeta_s=\zeta_s(B)=
\begin{cases}
\xi_{B^{(n_s^\ast)}} & \text{ if $n_s^\ast$ is even,} \\
\eta_{B^{(n_s^\ast)}} & \text{ if $n_s^\ast$ is odd.} 
\end{cases}
\]
Then the extended GK datum  of $B$, denoted by $\EGK(B)$, is defined as follows:
$$\EGK(B)=(n_1,\ldots,n_r;m_1,\ldots,m_r;\zeta_1,\ldots,\zeta_r).$$
Here, the integers $n_i$'s and $m_j$'s are obtained from $\mathrm{GK}(B)=(a_1, \cdots, a_n)$ as in Definition \ref{def3.0}.

\item For $B \in \calhnd_n(\frko )$, we define $\EGK(B)=\EGK(B')$, where $B'$ is an optimal form equivalent to $B$.
This definition does not depend on the choice of an optimal form $B'$ by Theorem \ref{thm:2.4}.

\item If $B$ is a symmetric matrix with respect to an ordered basis of  a quadratic lattice $(L, Q_L)$, then  $\mathrm{EGK}(L)$, called the extended GK datum of  $(L, Q_L)$, is defined by $\mathrm{EGK}(B)$.
\end{enumerate}
\end{definition}

Clearly, $\EGK(B)$  (or $\EGK(L)$) depends only on the isomorphism class of $B$ by Theorem \ref{thm:2.4}.

\subsection{Definition of local density}
\label{sec:4}
In this subsection, we explain a definition of the local density in the general case, taken from Section 5 of \cite{IK2}.

For $m\geq n\geq 1$, $A\in\calh_m(\frko)$, $B\in\calh_n(\frko)$, we put
\[
\cala_N(B, A)=
\{X=(x_{ij})\in\mathrm{M}_{mn}(\frko)/\frkp^N \mathrm{M}_{mn}(\frko)\;|\; A[X]-B\in\frkp^N\calh_n(\frko)\}.
\]
Then the local density $\alp(B, A)$ is defined by
\[
\alp(B,A)=\lim_{N\rightarrow \infty} (q^N)^{-mn+\frac{n(n+1)}2}\, \sharp\cala_N(B, A).
\]
Here, if  $N>2\ord(D_B)$, then the value
\[
(q^N)^{-mn+\frac{n(n+1)}2}\, \sharp\cala_N(B, A)
\]
does not depend on $N$.

Equivalently, 
 we have
\[
\alp(B, A)=\int_{y\in\mathrm{Sym}_n(F)} \int_{x\in\mathrm{M}_{mn}(\frko)} \psi\left(\mathrm{tr}( y(A[x]-B))\right)\, dx\, dy
\]
for  an additive character $\psi$ of $F$ with order $0$.
Here,  an additive character $\psi$ of $F$ with order $0$ means that (cf. Section 2 of \cite{IK2})
\[
\mathfrak{o}=\{a\in F| \psi(ax)=1 \textit{ for any }x\in \mathfrak{o}  \}.
\]
Here the integral $\int_{y\in \mathrm{Sym}_n(F)}$ with respect to $y\in \mathrm{Sym}_n(F)$ should be interpreted 
as the principal value integral
\[
\lim_{N\rightarrow \infty} \int_{y\in \vpi^{-N}\call}
\]
for some fixed lattice $\call\subset \mathrm{Sym}_n(F)$.

For the quadratic lattice $(L, Q_L)$ represented by  $B\in\calhnd_n(\frko)$ and the quadratic space $(V, Q_V)$ such that  $V=V\otimes_{\mathfrak{o}}F$ and $Q_V=Q_L\otimes 1$, 
the orthogonal group $G=\mathrm{O}_{F}(V)$ is an algebraic group defined over $F$.
The local density  for the single quadratic lattice $(L, Q_L)$, denoted by $\beta(L)$ or $\beta(B)$,  is defined by
\[
\beta(L)=\frac{1}{[G:G^\circ]}\alp(B, B).
\]
Here, $G^\circ$ is the identity component of $G$.


\begin{remark}
\begin{enumerate}
\item 
The local density can also be described in terms of a volume of certain $p$-adic manifold using scheme theory. 
In this remark, we will explain a main idea of this method following Section 3.1 of \cite{CY}.
For more detailed explanation, we refer to Section 3 of \cite{GY} or Section 3.1 of \cite{CY}. 

For the quadratic lattice $(L, Q_L)$   represented by   $B\in\calhnd_n(\frko)$ and the quadratic space $(V, Q_V)$ such that  $V=V\otimes_{\mathfrak{o}}F$ and $Q_V=Q_L\otimes 1$, 
regarding $\mathrm{M}_n(F)$ and $\mathrm{Sym}_n(F)$ as varieties over $F$, let $\omega_{M_n}$ and $\omega_{\mathcal{H}_n}$ be nonzero, translation-invariant top-degree forms on $\mathrm{M}_n(F)$ and $\mathrm{Sym}_n(F)$, respectively, with normalizations 
$$\int_{\mathrm{M}_n(\mathfrak{o})}|\omega_{M_n}|=1 \mathrm{~and~}  \int_{\mathcal{H}_n(\mathfrak{o})}|\omega_{\mathcal{H}_n}|=1.$$

We define a map $\rho : \mathrm{GL}_n(F) \rightarrow \mathrm{Sym}_n(F)$ by $\rho(X)= Q_V\circ X$.
Here we identify $\mathrm{GL}_n(F)=\mathrm{Aut}_F(V)$ 
and $\mathrm{Sym}_n(F)$ to be the set of (possibly degenerate) quadratic forms on $V$.
Then the inverse image of $Q_V$, along the map $\rho$, is $\mathrm{O}_F(V)$, which represents the group of $F$-linear self-maps of $V$ preserving the  quadratic space $(V, Q_V)$. One can also show that the morphism $\rho$ is representable as a morphism of schemes over $F$ and that this (necessarily unique) morphism is smooth. Then we have 
 a differential form $\omega_{L}$ on $\mathrm{M}_n(F)$ such that $\omega_{M_n}|_{\mathrm{GL}_n(F)}=
\rho^{\ast}\omega_{\mathcal{H}_n}\wedge \omega_{L}$. 
We denote by $\omega_{L}^{\mathrm{ld}}$  the restriction of $\omega_{L}$ to $\mathrm{O}_F(V)$. 
We sometimes write $\omega_{L}^{\mathrm{ld}}=\omega_{M_n}/\rho^{\ast}\omega_{\mathcal{H}_n}$.
Then we have
\[
\beta(L)=\frac{1}{[G:G^\circ]}
\int_{\mathrm{O}_{\mathfrak{o}}(L)}|\omega_L^{\mathrm{ld}}|.
\]
Here, $\mathrm{O}_{\mathfrak{o}}(L)$ represents the group, as a subset of $M_n(\mathfrak{o})$, preserving the quadratic lattice $(L, Q_L)$.

This equation is explained in Lemma 3.4 of \cite{GY}. 
For arbitrary $B  \in\calhnd_n(\frko)$  and $A  \in\calhnd_m(\frko)$,
 there is also a similar formulation of $\alp(B,A)$ as the integral of a volume form on a certain $p$-adic manifold. See Lemma 3.2 of \cite{CY} for more details.

\item The local density given in \cite{Cho} , which we denote in this paper by $\beta^C(L)$ (in \cite{Cho} this term is denoted by $\beta(L)$), uses a different normalization.  
With the above setting,
 let $\omega_{Sym_n}$ be a nonzero, translation-invariant  top-degree  form on $\mathrm{Sym}_n(F)$ with normalization
$$ \int_{\mathrm{Sym}_n(\mathfrak{o})}|\omega_{Sym_n}|=1.$$
Then $\beta^C(L)$ is defined to be 
\[
\beta^C(L)=\frac{1}{[G:G^\circ]}\int_{\mathrm{O}_{\mathfrak{o}}(L)}|\omega_{L}^{C, \mathrm{ld}}|,
\]
where 
 $\omega_{L}^{C, \mathrm{ld}}:=\omega_{M_n}/\rho^{\ast}\omega_{Sym_n}$.

\item In order to compare the both local densities $\beta(L)$ and $\beta^C(L)$, it suffices to compare the two volume forms $\omega_{\mathcal{H}_n}$ and $\omega_{Sym_n}$. 
Our normalizations imply that
\[
\vpi^{n(n-1)/2}\cdot \omega_{Sym_n}=  \omega_{\mathcal{H}_n}.
\]
Thus $\omega_{L}^{C, \mathrm{ld}}=\vpi^{n(n-1)/2}\cdot \omega_{L}^{\mathrm{ld}}$.
This  yields: 
\[
\beta^\mathrm{C}(L)=q^{-e n(n-1)/2} \beta(L).
\]
Here, $e$ is the ramification index of $F$ over $\QQ_2$.
\end{enumerate}

\end{remark}


\section{The local density formula of a single quadratic lattice}\label{sec:5.1}
In this section, we recall the local density formula for $\beta^C(L)$ 
  given in \cite{Cho}.
We assume that $F$ is an unramified finite field extension of $\QQ_2$.
We follow the formulation of \cite{Cho}.
Let
$(L, Q_L)$ be the quadratic lattice  represented by 
$B\in \calhnd_n(\frko)$.
We first collect necessary terminology below.
\begin{enumerate}
\item Recall that the bilinear form $\langle x, y \rangle_{Q_L}$ is defined by
\[
\langle x, y\rangle_{Q_L}=\frac12 (Q_L(x+y)-Q_L(x)-Q_L(y)).
\]
\item The scale  $\bfs(L)$ and the norm $\bfn(L)$
 are defined by
\begin{align*}
\bfs(L)=&\{\langle x, y\rangle_{Q_L}\;|\; x, y\in L\}
, \\
\bfn(L)=&\textit{the fractional ideal generated by $\{Q_L(x)\;|\; x\in L\}$}.
\end{align*}
\item The dual lattice $L^\sharp$ is defined by
\[
L^\sharp=\{x\in L\otimes F\;|\; \langle x, L\rangle_{Q_L}\subset \frko\}.
\]
\item $L$ is called a unimodular lattice if $L=L^\sharp$.
\item A unimodular lattice $L$ is \textit{of parity type I} if $\bfn(L)=\frko$ otherwise \textit{of parity type II}.

\item $(L, Q_L)$ is $i$-\textit{modular} if $(L, a^{-1}Q_L)$ is unimodular for some $a\in \mathfrak{o} \backslash \{0\}$ such that the exponential valuation of $a$ is $i$ 
(such an  $a$ being unique up to a unit).
In this case the parity type of $(L, Q_L)$ is defined to be the parity type of $(L, a^{-1}Q_L)$.
The zero lattice is considered to be \textit{of parity type II}.

\item Let $B(L)$ be the sublattice of $L$ such that $B(L)/2L$ is the kernel of the linear form $2^{-s(L)}Q_L$ mod 2 on $L/2L$.
Here, $s(L)$ is the integer such that $\bfs(L)=(2^{s(L)})$.
\item Let $Z(L)$ be the sublattice of $L$ such that $Z(L)/2L$ is the kernel of the quadratic form $2^{-s(L)-1}Q_L$ mod 2 on $B(L)/2L$.

\end{enumerate}

Let 
\[
L=\bigoplus_i L_i
\]
be a Jordan splitting.
We assume $\bfs(L_i)=(2^i)$,  allowing $L_i$ to be the zero lattice.
Put $n_i=\mathrm{rank}_\frko L_i$ and 
\[
\left\{
  \begin{array}{l}
 \textit{$A_i=\{x\in L \mid \langle x,L\rangle_{Q_L} \in 2^i\frko\}=L\cap 2^i  L^\sharp$};\\
 \textit{$B_i=B(A_i)$};\\
 \textit{$Z_i=Z(A_i)$}. 
    \end{array} \right.
\]
Then $Z_i$ is the sublattice of $B_i$ such that $Z_i/2 A_i$ is the kernel of the quadratic form $\frac{1}{2^{i+1}}Q_L$ mod 2 on $B_i/2 A_i$.
Let $\bar{V_i}=B_i/Z_i$ and $\bar{q}_i$ denote the nonsingular quadratic form $\frac{1}{2^{i+1}}Q_L$ mod 2 on $\bar{V_i}$.\\

 We assign a type to each $L_i$ as follows:
\[\left \{
  \begin{array}{l l}
   I & \quad  \textit{if $L_i$ is of parity type I,}\\
   I^o & \quad \textit{if $L_i$ is of parity type I and the rank of $L_i$ is odd,}\\
  I^e & \quad \textit{if $L_i$ is of parity type I and the rank of $L_i$ is even,}\\
II & \quad  \textit{if $L_i$ is of parity type II}.\\
    \end{array} \right.\]
    In addition, we say that $L_i$ is
\[\left \{
  \begin{array}{l l}
  \textit{bound} & \quad  \textit{if at least one of  $L_{i-1}$ or $L_{i+1}$ is of parity type I,}\\
 \textit{free} & \quad \textit{if both  $L_{i-1}$ and $L_{i+1}$ are of parity type II}.\\
    \end{array} \right.\]
Assume that a lattice $L_i$ is \textit{free} \textit{of type} $\textit{I}^e$.
We denote by $\bar{V_i}$ the $\frkk$-vector space $B_i/Z_i$.
Then we say that $L_i$ is
\[
 \left\{
  \begin{array}{l l}
  \textit{of type I}^e_1   & \quad  \text{if the dimension of  $\bar{V_i}$ is odd, 
}\\
  \textit{of type I}^e_2   &   \quad  \text{otherwise}.\\
    \end{array} \right.
\]
Notice that for each $i$, the type of $L_i$ is independent of the choice of a Jordan splitting.\\

Let $\underline{G}$ be the smooth integral model of $G=\mathrm{O}_{Q_F}$.
The readers are referred to the beginning of Section 3 of \cite{Cho} for a detailed  definition of $\underline{G}$.
The special fibre of $\underline{G}$ is denoted by $\tilde G$.
Then there exists a surjective morphism $\varphi$ (cf. Theorem 4.1 in \cite{Cho})
\[
\varphi=\prod_i \varphi_i : \tilde{G} ~ \longrightarrow  ~\prod_i \mathrm{O}(\bar{V_i}, \bar{q_i})^{\mathrm{red}}.
\]
The image  Im $\varphi_i$ is described as follows (cf. Remark 4.3  in \cite{Cho}).
     \[
      \begin{array}{c|c}
      \hline
      \mathrm{Type~of~lattice~}  L_i & \mathrm{Im~}  \varphi_i \\
      \hline
      \textit{I}^o,\ \  \mathrm{\textit{free}} & \mathrm{O}(n_i-1, \bar{q_i})\\
      \textit{I}^e_1,\ \  \mathrm{\textit{free}} &\mathrm{SO}(n_i-1, \bar{q_i})\\
      \textit{I}^e_2,\ \  \mathrm{\textit{free}} &\mathrm{O}(n_i-2, \bar{q_i})\\
      \textit{II},\ \  \mathrm{\textit{free}} &\mathrm{O}(n_i, \bar{q_i})\\
      \textit{I}^o,\ \  \mathrm{\textit{bound}} &\mathrm{SO}(n_i, \bar{q_i})\\
      \textit{I}^e,\ \  \mathrm{\textit{bound}} &\mathrm{SO}(n_i-1, \bar{q_i})\\
      \textit{II},\ \  \mathrm{\textit{bound}} &\mathrm{SO}(n_i+1, \bar{q_i}) \\
      \hline
      \end{array}
    \]

Let
 \begin{itemize}
 \item $\alpha$ be the number of all $i$ such that $L_i$ is \textit{free} \textit{of type} $\textit{I}^e_1$.
 \item  $\beta$ be the number of all $j$ such that $L_j$ is \textit{of type I} and $L_{j+2}$ is \textit{of type II}.
\end{itemize}

\begin{theorem}[\cite{Cho}, Theorem 4.12]\label{thmcho412}
We have an isomorphism
\[
\tilde{G}/R_u\tilde{G}\simeq 
\prod_i \mathrm{O}(\bar{V_i}, \bar{q_i})^{\mathrm{red}} \times (\mathbb{Z}/2\mathbb{Z})^{\alpha+\beta}.
\]
Here, $R_u\tilde{G}$  is the connected unipotent radical of $\tilde{G}$.
\end{theorem}

Let
\begin{itemize}
  \item $b$ be the total number of pairs of adjacent constituents $L_i$ and $L_{i+1}$ that are both \textit{of type I}.
  \item $c$ be the sum of ranks of all nonzero Jordan constituents $L_i$ that are \textit{of type} $\textit{II}$.
\end{itemize}

\begin{theorem}[\cite{Cho}, Theorem 5.2]\label{thmcho52}
The local density of $(L,Q_L)$ defined in \cite{Cho}, which we are denoting in this paper by $\beta^\mathrm{C}(L)$, is
\[
\beta^\mathrm{C}(L)=\frac{1}{[G:G^{\circ}]}q^N \cdot q^{-\mathrm{dim} G} \sharp\tilde{G}(\frkk),
\]
where 
\begin{align*}
N
=&t+\sum_{i<j} i\cdot n_i\cdot n_j+\sum_i i\cdot n_i\cdot (n_i+1)/2-b+c, \\
t =&\text{ the total number of $L_i$'s that are \textit{of type I}}.
\end{align*}
\end{theorem}

In the above local density formula, 
\[
\sharp\tilde{G}(\frkk)=\sharp R_u\tilde{G}(\frkk)\cdot \sharp (\tilde{G}/R_u\tilde{G})(\frkk).
\]

\section{Reformulation of the local density formula}\label{sec:5.2}
We are now ready to explain our main result.
In this section, we show that the local density $\beta(L)$ is determined by  a series of  Gross-Keating invariants and the (truncated) extended GK datums (cf. Theorem \ref{thm-ldgk}). 
We keep assuming that $F$ is unramified over $\mathbb{Q}_2$.
Let $(L, Q_L)$ be the quadratic lattice  represented by 
$B\in \calhnd_n(\frko)$.

\subsection{Reduced form of Ikeda and Katsurada}\label{sss1}
 In \cite{IK1}, Ikeda  and Katsurada introduced the notion of a  `reduced form' associated to $B$ and  showed it to be optimal.
We use a reduced form several times in this paper and thus
 provide its detailed definition through  Definitions \ref{def3.1}-\ref{def3.2}.
 They are taken from \cite{IK1}  for synchronization.
The main result of this subsection   is Proposition \ref{propi-ii}, which will be used in the next subsection.

Let $\mathfrak{S}_n$ be the symmetric group of degree $n$.
Let $\sigma\in \mathfrak{S}_n$ be an involution i.e. $\sigma^2=id$.
For a non-decreasing sequence of non-negative integers $\underline{a}=(a_1, \cdots, a_n)$, 
we set
\[\mathcal{P}^0=\mathcal{P}^0(\sigma)=\{i|1\leq i \leq n, i=\sigma(i)\}, \]
\[\mathcal{P}^+=\mathcal{P}^+(\sigma)=\{i|1\leq i \leq n, a_i>a_{\sigma(i)}\}, \]
\[\mathcal{P}^-=\mathcal{P}^-(\sigma)=\{i|1\leq i \leq n, a_i<a_{\sigma(i)}\}. \]

\begin{definition}[\cite{IK1}, Definition 3.1]\label{def3.1}
We say that an involution $\sigma\in \mathfrak{S}_n$ is  $\underline{a}$-admissible  if the following three conditions are satisfied:
\begin{itemize}
\item[(i)] 
$\calp^0$ has at most two elements.
If $\calp^0$ has two distinct elements $i$ and $j$, then $a_i\not\equiv a_j \text{ mod $2$}$.
Moreover, if $i \in   \calp^0$, then 
\[
a_i=\max\{ a_j \, |\, j\in \calp^0\cup\calp^+, \, a_j\equiv a_i \text{ mod }2\}.
\] 
\item[(ii)]
For $s=1, \ldots, r$, we have
\[
\#(\mathcal{P}^+\cap I_s)\leq 1, ~~~~~~\textit{   }~~~~~~~~
\#(\mathcal{P}^-\cap I_s)+\#(\mathcal{P}^0\cap I_s)\leq 1.
\]
Here, $I_s$ is defined in Definition \ref{def3.0}.

\item[(iii)]
If $i \in   \calp^-$, then 
\[
a_{\sig(i)}=\min\{a_j \,| \, j\in \calp^+, a_j>a_i,\, a_j\equiv a_i \text{ mod } 2\}.
\]
Similarly, 
if $i \in   \calp^+$, then 
\[
a_{\sig(i)}=\max\{a_j \,| \, j\in \calp^-, a_j<a_i,\, a_j\equiv a_i \text{ mod } 2\}.
\]
\end{itemize}
If $\sig$ is an $\ua$-admissible involution, the pair $(\ua, \sig)$ is called a  GK type. 
\end{definition}

\begin{definition}[\cite{IK1}, Definition 3.2]\label{def3.2}
Write $B=\begin{pmatrix}b_{ij}\end{pmatrix}\in \calhnd_n(\frko)$.
Let $\underline{a}\in S(B)$ (cf. Definition \ref{def:2.1}.(1)).
Let $\sigma\in \mathfrak{S}_n$ be an $\underline{a}$-admissible involution.
We say that $B$ is a reduced form of  GK-type $(\underline{a}, \sigma)$ if the following conditions are satisfied:
\begin{enumerate}
\item If $i \notin \mathcal{P}^0$, $j=\sigma(i)$, and $a_i\leq a_j$, then
\[\mathrm{GK}\begin{pmatrix}\begin{pmatrix}b_{ii} & b_{ij}\\ b_{ij}&b_{jj}\end{pmatrix}\end{pmatrix}=(a_i, a_j).\]
Note that this condition is equivalent to the following condition (by Proposition 2.3 of \cite{IK1}).
\[\left\{
  \begin{array}{l l}
  \mathrm{ord}(2b_{ij})=\frac{a_i+a_{j}}{2}   & \quad    \text{if $i\notin \mathcal{P}^0$, $j=\sigma(i)$};\\
   \mathrm{ord}(b_{ii})=a_i   & \quad    \text{if $i\in \mathcal{P}^-$}.
    \end{array} \right.\]

\item if $i\in \mathcal{P}^0$, then
\[\mathrm{ord}(b_{ii})=a_i.\]

\item If $j\neq i, \sigma(i)$, then
\[\mathrm{ord}(2b_{ij})>\frac{a_i+a_j}{2}.\]
\end{enumerate}

Saying just	`reduced form' without an $\underline{a}$ or a $\sigma$ means `reduced form' of GK-type $(\underline{a}, \sigma)$ for some non-increasing sequence $\underline{a}$ of integers and an $\underline{a}$-admissible involution $\sigma$. 
\end{definition}

\begin{theorem}[\cite{IK1}, Corollary 5.1]\label{thm5.1}
A reduced form is optimal.
More precisely, if $B$ is a reduced form of  GK-type $(\underline{a}, \sigma)$, then
$$\mathrm{GK}(B)=\underline{a}.$$
\end{theorem}

We list a few facts about the above definitions.
\begin{remark}\label{rmk1}
\begin{enumerate}
\item For any given non-decreasing sequence of non-negative integers $\underline{a}=(a_1, \cdots, a_n)$,
there always exists an $\underline{a}$-admissible involution (cf. the paragraph following Definition 3.1 of \cite{IK1}).

\item For  $B \in \calhnd_n(\frko)$,
there always exist  a  $\mathrm{GK}(B)$-admissible involution $\sigma$ and
a reduced form of  GK type $(\mathrm{GK}(B), \sigma)$, which is equivalent to $B$
(cf. Theorem 4.1 of \cite{IK1}).
Equivalently, for a quadratic lattice $(L, Q_L)$, there always exists a reduced form which represents the integral quadratic form $Q_L$.

Such an involution $\sigma$ is not
unique, but it is so up to a certain notion of equivalence, as will be recalled in Remark \ref{rmk11}.

\item The first integer of $\mathrm{GK}(L)$ is 
the exponential valuation of a generator of $\bfn(L)$ (cf. Lemma B.1 of \cite{thyang}).
\end{enumerate}
\end{remark}

\begin{remark}\label{rmk11}
We say that two $\underline{a}$-admissible involutions are equivalent if they are conjugate by an element of $\mathfrak{S}_{n_1}\times \cdots \times \mathfrak{S}_{n_r}$.
Here, we follow the notation introduced in Definition \ref{def3.0} to specify the integers $n_1, \cdots, n_r$.
If $\sigma$ is an $\underline{a}$-admissible involution, then the equivalence class of $\sigma$ is determined by 
\begin{equation}\label{eqset}
\#(\mathcal{P}^+\cap I_s), ~~\textit{  }~~~ \#(\mathcal{P}^-\cap I_s),  ~~\textit{  }~~~  \#(\mathcal{P}^0\cap I_s)
\end{equation}
for $1\leq s \leq r$ (cf. the paragraph following Remark 4.1 in \cite{IK1}).

Let $\sigma$ and $\tau$ be  $\mathrm{GK}(B)$-admissible involutions 
associated to reduced forms of  GK types $(\mathrm{GK}(B), \sigma)$ and $(\mathrm{GK}(B), \tau)$, respectively, which are equivalent to a given symmetric matrix $B$.
 Then $\sigma$ and $\tau$ are equivalent (cf. Theorem 4.2 of \cite{IK1}).
 Therefore, the above sets in (\ref{eqset}) for $B$ are independent of the choice of a 
 $\mathrm{GK}(B)$-admissible involution with a reduced form.
\end{remark}



\begin{lemma}\label{lemgk}
  If $\mathrm{GK}(B)=(a_1, \cdots, a_n)$, then
$$\mathrm{GK}(2^lB)=(a_1+l, \cdots, a_n+l).$$
\end{lemma}
\begin{proof}
We write $\mathrm{GK}(2^lB)=(b_1, \cdots, b_n)$. 
It is obvious by Definition \ref{def:2.1} that 
$(b_1, \cdots, b_n) \succeq (a_1+l, \cdots, a_n+l)$.
Since $\mathrm{GK}(B)=\mathrm{GK}(2^{-l}\cdot 2^lB)$, 
we also have $(a_1, \cdots, a_n)\succeq (b_1-l, \cdots, b_n-l)$.
Thus we have 
$(b_1, \cdots, b_n) = (a_1+l, \cdots, a_n+l)$.
\end{proof}

Let $\sigma\in \mathfrak{S}_n$ be an $\underline{a}$-admissible involution
and let $\tau \in \mathfrak{S}_m$ be a  $\underline{b}$-admissible involution.
We choose embeddings of $\underline{a}$ and $\underline{b}$ into $\underline{a} \cup \underline{b}$.
Here, the notion of $\underline{a} \cup \underline{b}$ is given at the beginning of Section \ref{ssnota}.
The involution $\sigma\cup \tau$ is defined as the element in $\mathfrak{S}_{n+m}$ such that the restriction of $\sigma\cup \tau$ to $\underline{a}$ (resp. $\underline{b}$)
along the embedding is well-defined and is the same as $\sigma$ (resp. $\tau$).
If we assume that both $\mathcal{P}^0(\sigma)$ and $\mathcal{P}^+(\sigma)$  are empty (thus $\mathcal{P}^-(\sigma)$ is empty as well),
i.e. $\sigma(i)\neq i$  and $a_i=a_{\sigma(i)}$ for any $1\leq i \leq n$, 
then it is easy to see that $\sigma\cup \tau$ is an  $\underline{a}\cup \underline{b}$-admissible involution
for any pair of embeddings from $\underline{a}$ and $\underline{b}$ into $\underline{a} \cup \underline{b}$.

\begin{lemma}\label{red}
Let $B=X\bot Y$ be of size $(n+2)\times (n+2)$, where $X=2^l\begin{pmatrix} 2u& w \\ w & 2v  \end{pmatrix}$
with  $w \in \mathfrak{o}^{\times}$ and  $u, v \in \mathfrak{o}$.
Then  \[\mathrm{GK}(B)=\mathrm{GK}(X)\cup \mathrm{GK}(Y).\]
\end{lemma}
\begin{proof}
Note that $\mathrm{GK}(X)=(l+1, l+1)$ by Proposition 2.3 of \cite{IK1} and Lemma \ref{lemgk}.
Let $\underline{a}=\mathrm{GK}(X)$ and let $\sigma$ be the associated non-trivial  $\underline{a}$-admissible involution (i.e. $\sigma(1)=2$).
Then  $\mathcal{P}^0(\sigma)$ and $\mathcal{P}^+(\sigma)$ are empty and $X$ is a reduced form of   GK-type $(\underline{a}, \sigma)$.

Let $Y'$ be a reduced form  of   GK-type $(\underline{b}, \tau)$ which is equivalent to $Y$, where $\underline{b}=\mathrm{GK}(Y)$.
The existence of a reduced form is guaranteed by  Remark \ref{rmk1}.(2).
The argument explained just before this lemma yields that $\sigma\cup \tau$ is an   $\underline{a}\cup \underline{b}$-admissible involution.
Let $(M, Q_M)$ be the quadratic lattice  represented by  the symmetric matrix $X\bot Y'$.
We choose  $(e_1, \cdots, e_{n+2})$ to be a basis of $M$ 
such that with respect to this basis the symmetric matrix of the quadratic lattice $(M, Q_M)$ is  $X\bot Y'$.

Recall that $\underline{a}\cup \underline{b}$ is a reordered multiset of $\{a_1, a_2, b_1, \cdots, b_n\}$, where 
$\underline{a}=\{a_1, a_2\}$ and $\underline{b}=\{b_1, \cdots, b_n\}$.
Let $\varphi$ be the permutation such that $\varphi(a_1, a_2, b_1, \cdots, b_n)=\underline{a}\cup \underline{b}$.
Here we consider $(a_1, a_2, b_1, \cdots, b_n)$ as an ordered multiset.
We define $(e_1', \cdots, e_{n+2}')$ to be the reordered basis of $M$ such that 
$(e_1', \cdots, e_{n+2}')=\varphi\left(e_1, \cdots, e_{n+2}\right)$.
Then the  symmetric matrix of $M$  with respect  to the reordered basis $(e_1', \cdots, e_{n+2}')$, which is equivalent to $X\bot Y'$,
is a reduced form of  GK-type $(\underline{a}\cup \underline{b},  \sigma\cup \tau)$ by Definition \ref{def3.2}.
The lemma then follows from Theorem \ref{thm5.1}.
\end{proof}

In general, let $B=\oplus B_i$ be a Jordan splitting such that $B_i$ is $i$-modular of size $n_i\times n_i$.
By this, we mean that $B$ is an orthogonal sum of $B_i$'s and each $B_i$ is of the form $2^iB_i'$,
where $B_i'$ is unimodular, i.e.
 all entries of $B_i'$ are elements in  $\mathfrak{o}$ and the determinant of $B_i'$ is a unit in $\mathfrak{o}$.
Each unimodular symmetric matrix $B_i'$ is of one of the following forms (cf. Theorem 2.4 of \cite{Cho}):
\[ \left\{
  \begin{array}{l l}
    (\bigoplus_k \begin{pmatrix} 2a_k& u_k \\ u_k & 2b_k  \end{pmatrix})    & \quad    \textit{: type II};\\
  (\bigoplus_k \begin{pmatrix} 2a_k& u_k \\ u_k & 2b_k  \end{pmatrix}) \oplus (\epsilon)  & \quad    \textit{: type $I^o$};\\
  (\bigoplus_k \begin{pmatrix} 2a_k& u_k \\ u_k & 2b_k  \end{pmatrix}) \oplus (\epsilon) \oplus (\epsilon')   & \quad    \textit{: type $I^e$}.
    \end{array} \right.\]
Here, $a_k, b_k \in \mathfrak{o}$ and $u_k, \epsilon, \epsilon' \in \mathfrak{o}^{\times}$.
Then we have the following reduction formula about $\mathrm{GK}(B)$ by using  Lemma \ref{red} inductively.

\begin{proposition}\label{propi-ii}
Let $B=\oplus B_i$ be a Jordan splitting such that $B_i$ is $i$-modular of size $n_i\times n_i$.
By using the argument explained in the paragraph just before this proposition, 
we write $B_i=B_i^{\dag}\bot B_i^{\ddag}$ such that $B_i^{\dag}$ is \textit{of type II} and $B_i^{\ddag}$ is
 empty (if $B_i$ is \textit{of type} $II$), of rank 1 (if $B_i$ is \textit{of type} $I^o$),
or of rank 2 (if $B_i$ is \textit{of type} $I^e$).
Then  \[\mathrm{GK}(B)=\mathrm{GK}(\oplus B_i^{\dag})\cup \mathrm{GK}(\oplus B_i^{\ddag})\]
and
\begin{multline*}
\mathrm{GK}(\oplus B_i^{\dag})=(\bigcup_{\textit{$L_i$:of type $II$}}(i+1, i+1)^{n_i/2}) \cup\\
(\bigcup_{\textit{$L_i$:of type $I^0$}}(i+1, i+1)^{(n_i-1)/2})\cup (\bigcup_{\textit{$L_i$:of type $I^e$}}(i+1, i+1)^{(n_i-2)/2}).   
\end{multline*}
Here, $(i+1, i+1)^{n_i/2}=\cup_{n_i/2}(i+1, i+1)$ and so on.
\end{proposition}

\begin{proof}
Since $B_i^{\dag}$ is $i$-modular \textit{of type II},  it is equivalent to an orthogonal sum of $2\times 2$ matrices of the form
$2^i\begin{pmatrix} 2a& u \\ u & 2b  \end{pmatrix}$ with  $u \in \mathfrak{o}^{\times}$ and  $a, b \in \mathfrak{o}$
as explained in the paragraph just before this proposition.
Then the proposition follows from  Lemma \ref{red} inductively.
\end{proof}

\subsection{Description in terms of $\GK(L\oplus -L)$}\label{sss2}
Let $(-L, Q_{-L})$ be the quadratic lattice  represented by 
$-B\in \calhnd_n(\frko)$.
Let 
\[
L=\bigoplus_i L_i
\]
be a Jordan splitting such that $\bfs(L_i)=(2^i)$,  allowing some of the $L_i$ to possibly be the zero lattice.
Put $n_i=\mathrm{rank}_\frko L_i$.
Then 
\[
L\oplus -L=\bigoplus_i (L_i\oplus -L_i)
\]
is a Jordan splitting of  $L\oplus -L$ such that $\bfs(L_i\oplus -L_i)=(2^i)$.

The Gross-Keating invariant of $L\oplus -L$ is computed as follows:
\begin{proposition}\label{propgkl-l}
We have that
\[
\mathrm{GK}(L\oplus -L)=\bigcup_i\mathrm{GK}(L_i\oplus -L_i).
\]
Here, 
\[
\mathrm{GK}(L_i\oplus -L_i)=\left\{
  \begin{array}{l l}
 (\underbrace{i+1, \cdots, i+1}_{2n_i})   & \quad  \textit{if $L_i$ is of type II};\\
 (\underbrace{i+1, \cdots, i+1}_{2n_i-2})\cup (i, i+2)   & \quad  \textit{if $L_i$ is of type $I$}.
    \end{array} \right.
\]
\end{proposition}
\begin{proof}
If $L_i$ is of type $I$ (resp. $II$), then $L_i\oplus -L_i$ is of type $I^e$ (resp. $II$).
If $L_i$ is nonzero, then we  claim that there is  an ordered basis of $L_i\oplus -L_i$
such that with respect to this basis the symmetric matrix of the quadratic lattice $L_i\oplus -L_i$ is 
\begin{equation}\label{eq42}
\left\{
  \begin{array}{l l}
2^i(\bigoplus_k \begin{pmatrix} 2a_k& u_k \\ u_k & 2b_k  \end{pmatrix})    & \quad  \textit{if $L_i$ is of type II};\\
2^i(\bigoplus_k \begin{pmatrix} 2a_k& u_k \\ u_k & 2b_k  \end{pmatrix})\oplus 
2^i \begin{pmatrix} 1& 1 \\ 1 & 4c_i  \end{pmatrix}     & \quad  \textit{if $L_i$ is of type $I$}
    \end{array} \right.
\end{equation}
with $c_i\in \mathfrak{o}$.
Here, $a_k, b_k \in \mathfrak{o}$ and $u_k \in \mathfrak{o}^{\times}$.
To prove the claim, we use Theorem 2.4 of \cite{Cho}.
Theorem 2.4 of loc. cit. directly verifies our claim when $L_i$ is of type $II$ and thus we may and do assume that $L_i$ is of type $I$.
If $L$ is of type $I^o$, then it suffices to prove that 
$\begin{pmatrix}
\epsilon&0\\0&-\epsilon
\end{pmatrix}$ with $\epsilon\in \mathfrak{o}^{\times}$ is equivalent to $\begin{pmatrix}
1&1\\1&4c
\end{pmatrix}$
for some $c\in \mathfrak{o}$.

We denote by $Q_{\epsilon}$ the quadratic lattice of rank $2$ represented by  the symmetric matrix $\begin{pmatrix}
\epsilon&0\\0&-\epsilon
\end{pmatrix}$.
Choose the ordered basis  $(a_1, a_2)$ of $Q_{\epsilon}$ such that with respect to this basis the symmetric matrix of $Q_{\epsilon}$ is 
 $\begin{pmatrix}
\epsilon&0\\0&-\epsilon
\end{pmatrix}$.
Then the matrix of $Q_{\epsilon}$ with respect to the ordered basis  $(a_1, a_1+a_2)$
is $\begin{pmatrix}
\epsilon&\epsilon\\\epsilon&0
\end{pmatrix}$.
Since any unit in $\mathfrak{o}$ is square modulo $2$, we may assume that $\epsilon \equiv 1$ mod $2$. 
By Theorem 2.4 of \cite{Cho}, $\begin{pmatrix}
\epsilon&\epsilon\\\epsilon&0
\end{pmatrix}$ is equivalent to $\begin{pmatrix} 1& 1 \\ 1 & 2c  \end{pmatrix}$ for some $c\in \mathfrak{o}$.
Thus it suffices to show that $c$ is contained in the ideal $(2)$.

Consider the following two quadratic forms: $f(x,y)=\epsilon x^2 + 2 \epsilon xy$ and $g(x,y)=x^2+2xy+2cy^2$.
These two quadratic forms are determined by two matrices  $\begin{pmatrix}
\epsilon&\epsilon\\\epsilon&0
\end{pmatrix}$ and $\begin{pmatrix} 1& 1 \\ 1 & 2c  \end{pmatrix}$, respectively, and thus  are equivalent.
Note that both  $f$ modulo $2$ and $g$ modulo $2$ define linear forms as $\frkk$-valued functions. 
We consider the kernels of these two linear forms respectively.
For example, the kernel of $f$ modulo $2$ is generated by $(2a_1, a_1+a_2)$.
It is easy to see that the restriction of $f$ to the kernel  is $4\epsilon x^2 + 4 \epsilon xy$, which we denote it by $f_{res}(x,y)$,
and that the restriction of  $g$ to the kernel is $4x^2+4xy+2cy^2$, which we denote it by $g_{res}(x,y)$.
Since $f$ and $g$ are equivalent,  $f_{res}$ is   equivalent to $g_{res}$ as well.
Thus, the norm of $f_{res}$, which is the ideal $(4)$, should also be the norm of  $g_{res}$.
This directly yields that $c$ is contained in the ideal $(2)$.


If $L$ is of type $I^e$, then
we may and do assume that the rank of $L_i\oplus -L_i$ is $4$. 
We choose the ordered basis $(a_1, a_2, a_3, a_4)$ of $L_i\oplus -L_i$
such that with respect to this basis the symmetric matrix of $L_i\oplus -L_i$ is as follows:
$$\begin{pmatrix}
1&1&0&0\\1&2\gamma&0&0\\
0&0&-1&-1\\0&0&-1&-2\gamma
\end{pmatrix}$$
with $\gamma \in \mathfrak{o}$.
Then the symmetrix matrix of $L_i\oplus -L_i$ with respect to the ordered basis $(a_1+a_3, a_2+2\gamma a_3, a_2+a_3, a_2+a_4)$ is
$\begin{pmatrix}
0&1-2\gamma&0&0\\1-2\gamma&2\gamma-4\gamma^2&0&0\\
0&0&-1+2\gamma&-1+2\gamma\\0&0&-1+2\gamma&0
\end{pmatrix}$.
Thus it suffices to prove that 
$\begin{pmatrix}
-1+2\gamma&-1+2\gamma\\-1+2\gamma&0
\end{pmatrix}$ is equivalent to $\begin{pmatrix}
1&1\\1&4c
\end{pmatrix}$
for some  $c\in \mathfrak{o}$.
The proof of this is the same as  the above case (when $L$ is of type $I^o$) and thus we may skip it.


Before proceeding our proof, 
we explain an involution defined on an ordered basis and a reordered basis.
For an involution $\sigma$ defined on the set $\{1, \cdots, n\}$,
 define the involution $\sigma$ on  an ordered  basis $(e_1, \cdots, e_n)$ of a lattice $L$ as follows:
 \[
 \textit{$\sigma(e_i)=e_j$ if $\sigma(i)=j$.}
 \] 
For an involution $\sigma$ defined on the ordered basis $(e_1, \cdots, e_n)$, if $(e_1', \cdots, e_n')$, which we denote it by $\mathcal{RE}$, is a reordered basis for $(e_1, \cdots, e_n)$, then the involution $\sigma_{\mathcal{RE}}$  on $(e_1', \cdots, e_n')$, which is induced from $\sigma$,  is defined as follows:
\[ \textit{$\sigma_{\mathcal{RE}}(e_r')=e_s'$ if $e_r'=e_i, e_s'=e_j, $ and $\sigma(e_i)=e_j$.}\] 
Then we define the involution $\sigma_{\mathcal{RE}}$ on $\{1, \cdots, n\}$ as follows:
\[
 \textit{$\sigma_{\mathcal{RE}}(r)=s$ if $\sigma_{\mathcal{RE}}(e_r')=e_s'$.}
\]

We claim that the Gross-Keating invariant of $L_i \oplus -L_i$ is as described in the statement of the proposition.
Let $(e_1^{(i)}, f_1^{(i)}, \cdots, e^{(i)}_{n_i}, f^{(i)}_{n_i})$ be an ordered basis of $L_i\oplus -L_i$ such that 
with respect to this basis the symmetric matrix of $L_i\oplus -L_i$ is described as in    (\ref{eq42}).
We consider the involution $\sigma$ which switches  $e^{(i)}_j$ and $f^{(i)}_j$.
Then it is easy to see that if $L_i$ is of type $II$, then the symmetric matrix  of $L_i\oplus -L_i$  with respect to the ordered basis $(e^{(i)}_1, f^{(i)}_1, \cdots, e^{(i)}_{n_i}, f^{(i)}_{n_i})$ is a reduced form of GK-type 
\[\left(\left(\underbrace{i+1, \cdots, i+1}_{2n_i} \right), \sigma \right).\]
It is also easy to see  that if $L_i$ is of type $I$, then the symmetric matrix  of $L_i\oplus -L_i$  with respect to the reordered basis $(e^{(i)}_{n_i}, e^{(i)}_1, f_1^{(i)}, \cdots, e^{(i)}_{n_i-1}, f^{(i)}_{n_i-1}, f^{(i)}_{n_i})$, which we denote it by $\mathcal{RE}^{(i)}$,  is a reduced form 
 of GK-type
  \[
  \left(\left(\underbrace{i+1, \cdots, i+1}_{2n_i-2}\right)\cup \left(i, i+2\right), \sigma_{\mathcal{RE}^{(i)}} \right).\] 

To prove our claim for $L\oplus -L$,
 we may and do assume that each $(L_i\oplus -L_i)$ is of type $I$ with rank $2$ or the zero lattice by Proposition \ref{propi-ii}.
We consider the ordered basis of $L\oplus -L$ as follows:
\[
\left(\cdots, \left(e_1^{(i-1)}, f_1^{(i-1)}\right), \left(e_1^{(i)}, f_1^{(i)}\right),  \left(e_1^{(i+1)}, f_1^{(i+1)}\right), \cdots\right).
\]
Here, $\left(e_1^{(i)}, f_1^{(i)}\right)$ is an ordered basis of $(L_i\oplus -L_i)$ such that the symmetric matrix of $(L_i\oplus -L_i)$ is described as in    (\ref{eq42}), if  $(L_i\oplus -L_i)$ is of type $I$ with rank $2$.
If  $(L_i\oplus -L_i)$ is the zero lattice, then we understand that $\left(e_1^{(i)}, f_1^{(i)}\right)$ is empty.
Let $\sigma$ be the involution on the above ordered basis which switches $e^{(i)}_{1}$ and  $f^{(i)}_{1}$.

We choose the reordered basis of $L\oplus -L$ in the following manner:
Let $j, k$ be  integers such that both $L_{j-1}\oplus -L_{j-1}$ and $L_{j+k+1}\oplus -L_{j+k+1}$ are the zero lattices, and all of $L_{j}\oplus -L_{j}, \cdots, L_{j+k}\oplus -L_{j+k}$ are non-zero lattices, where $k\geq 0$.
Recall that we consider the ordered basis of the lattice
$(L_{j}\oplus -L_{j})\oplus \cdots \oplus (L_{j+k}\oplus -L_{j+k})$ as
\[
\left(\left(e_1^{(j)}, f_1^{(j)}\right), \cdots, \left(e_1^{(j+k)}, f_1^{(j+k)}\right)\right).
\]

Then we choose the reordered basis, which we denote it by $\mathcal{RE}_{j,k}$, of  $(L_{j}\oplus -L_{j})\oplus \cdots \oplus (L_{j+k}\oplus -L_{j+k})$  as follows:
\[
\left(\left(e_1^{(j)}, e_1^{(j+1)}\right), \left( f_1^{(j)}, e_1^{(j+2)}\right), \left( f_1^{(j+1)}, e_1^{(j+3)}\right), \cdots, 
\left( f_1^{(j+k-2)}, e_1^{(j+k)}\right), \left( f_1^{(j+k-1)}, f_1^{(j+k)}\right)\right).
\]
Here, if $k=0$, then we understand the above reordered basis as $\left( e_1^{(j)}, f_1^{(j)}\right)$.
If $k=1$, then we understand the above reordered basis as $\left(e_1^{(j)}, e_1^{(j+1)}, f_1^{(j)}, f_1^{(j+1)}\right)$.
The reordered basis of $L\oplus-L$ is then defined by the ordered set 
$\bigcup\limits_{j,k}\mathcal{RE}_{j,k}$ with respect to $j$, which we denote it by $\mathcal{RE}$.

Then it is easy to see that the symmetric matrix of the quadratic lattice $L\oplus -L$ with respect to  the above  reordered basis $\mathcal{RE}$ is a reduced form 
of GK-type 
\[
\left(\bigcup_{i} (i, i+2), \sigma_{\mathcal{RE}}\right).
\]
This completes the proof.
\end{proof}

We write $\mathrm{GK}(L\oplus -L)$ as $(a_1, \cdots, a_{2n})$. 
If $L_i$ is of type $I$, then 
the involution $\sigma_{\mathcal{RE}}$, which is described in the proof of the above proposition, 
satisfies the following property that 
$\sigma_{\mathcal{RE}}(i')=j'$ for some $1\leq i', j'\leq 2n$ such that 
$a_{i'}=i$ and $a_{j'}=i+2$. 
Thus, by Remark \ref{rmk11},
 any  $\mathrm{GK}(L\oplus -L)$-admissible involution   associated to any reduced form of $L\oplus -L$ satisfies the same property.
Using this, we can recover the parity type and the rank of $L_i$ from $\mathrm{GK}(L\oplus -L)$, as stated in the following corollary.

\begin{corollary} 
\label{cortypeli}
Let $\mathrm{GK}(L\oplus -L)=(a_1, \cdots, a_{2n})$ 
 and let $\sigma$ be a $\mathrm{GK}(L\oplus -L)$-admissible involution associated to a reduced form of $L\oplus -L$.
For each $i\in \mathbb{Z}$, we define two numbers $\mathcal{A}_i$ and $\mathcal{B}_i$ as follows:
\[
\left\{
  \begin{array}{l}
 \mathcal{A}_i=\#\{t | a_t=a_{\sigma(t)}=i+1\};\\
 \mathcal{B}_i=\#\{t| a_t=i \textit{ and } 
 a_{\sigma(t)}=i+2\}.
     \end{array} \right.
\]
Here $\mathcal{A}_i$ is even (possibly zero) and $\mathcal{B}_i$ is  either $0$ or $1$.
Then  $\mathcal{A}_i$ and $\mathcal{B}_i$ determine the following information about $L_i$:
\[
\left\{
  \begin{array}{l l}
L_i : \textit{type II}, n_i=\mathcal{A}_i/2  & \quad    \textit{if $\mathcal{B}_i=0$};\\
 L_i : \textit{type $I^o$}, n_i=\mathcal{A}_i/2+1 & \quad   \textit{if $\mathcal{B}_i=1$ and $\mathcal{A}_i\equiv 0$ mod 4};\\
 L_i : \textit{type $I^e$}, n_i=\mathcal{A}_i/2+1 & \quad    \textit{if $\mathcal{B}_i=1$ and $\mathcal{A}_i\equiv 2$ mod 4}.
      \end{array} \right.
\]
\end{corollary}

Note that the two numbers $\mathcal{A}_i$ and $\mathcal{B}_i$ are  independent of the choice of an admissible involution $\sigma$ since any two admissible involutions are equivalent, as explained in  Remark \ref{rmk11}.


\begin{corollary}\label{cortypelii}
Let $\mathrm{GK}(L\oplus -L)=(a_1, \cdots, a_{2n})$ and let
$\mathcal{C}_i=\#\{t | a_t=i+1\}$.
Then $\mathcal{B}_i$ is determined by the parity type of $L_i$ as follows:
\[
\mathcal{B}_i= \left\{
  \begin{array}{l l}
0  & \quad  \textit{if  $L_{i}$ is of parity type II};\\
1  & \quad  \textit{if  $L_{i}$ is of parity type I}.      \end{array} \right.
\]
In addition,
we have the following description of $\mathcal{A}_i$   in terms of $\mathcal{C}_i$.
\[
\mathcal{A}_i= \left\{
  \begin{array}{l l}
\mathcal{C}_i  & \quad  \textit{if both $L_{i-1}$ and $L_{i+1}$ are of parity type II};\\
\mathcal{C}_i-1 & \quad  \textit{if exactly one of $L_{i-1}$ and $L_{i+1}$ is of parity type I};\\
\mathcal{C}_i-2  & \quad  \textit{if both $L_{i-1}$ and $L_{i+1}$ are of parity type I}.
      \end{array} \right.
\]

\end{corollary}

\begin{remark}\label{rmk2}
We note that 
$\EGK(L\oplus -L)$ 
does not determine 
 the local density $\beta(L)$ (and hence, neither does $\GK(L\oplus -L)$).
As an  example, let $L$ be the quadratic lattice represented by the symmetric matrix 
$\begin{pmatrix} 1&0\\0&1  \end{pmatrix}$
and let $M$ be the quadratic lattice represented by the symmetric matrix 
$\begin{pmatrix} 1&0\\0&3  \end{pmatrix}$.
Then $L$ is unimodular of type $I^e_1$ and $M$ is unimodular of type $I^e_2$
so that they have the different local densities.
But we can easily see that 
\[\EGK(L\oplus -L)=\EGK(M\oplus -M)=(1,2,1;0,1,2;1,1,1).\]
\end{remark}

\subsection{Truncated EGK}\label{sss3}
 Remark \ref{rmk2} implies that
we need the extended Gross-Keating datums of more quadratic lattices, and not just  $\mathrm{EGK}(L\oplus -L)$,
  to completely determine the local density $\beta(L)$. 
To do that, we
 consider the normalized quadratic lattice $(L\cap 2^i  L^\sharp, \frac{1}{2^i}Q_L)$ 
for each integer $i$ such that $L_i$ is nonzero, where 
$L=\bigoplus L_i$ is a Jordan splitting with  $\bfs(L_i)=(2^i)$.
Note that $L\cap 2^i  L^\sharp$ is denoted by $A_i$ in  Section \ref{sec:5.1}.
We can choose a Jordan splitting $L\cap 2^i  L^\sharp=\bigoplus_{k \geq 0} M_j$ with $\bfs(M_j)=(2^j)$ such that  $M_0=L_i$ (cf. Remark 2.8 of \cite{Cho}).
In this subsection and the next subsection, the quadratic lattice $L\cap 2^i  L^\sharp$, for each $i$ such that $L_i$ is nonzero, is meant to be the normalized quadratic lattice as described above.
In this subsection, we fix  $\mathrm{GK}(L\cap 2^i  L^\sharp)=(a_1, \cdots, a_n)$.
\begin{lemma}\label{lem48}
A reduced form $B^i$, which represents the  restriction of $\frac{1}{2^i}Q_L$ to $L\cap 2^i  L^\sharp$, is expressed as the block matrix: 
$$B^i=\begin{pmatrix}
B_{00}^i&B_{01}^i&B_{02}^i\\
B_{10}^i&B_{11}^i&B_{12}^i\\
B_{20}^i&B_{21}^i&B_{22}^i
\end{pmatrix}$$ 
 satisfying the following conditions:
\begin{enumerate}
\item the size of $B_{00}^i$ is the same as the number of $0$'s in $\mathrm{GK}(L\cap 2^i  L^\sharp)$;
 \item $B_{00}^i$ is a reduced form such that $\mathrm{GK}(B_{00}^i)$ consists of $0$'s;
\item $B_{11}^i$ is  a reduced form such that $\mathrm{GK}(B_{11}^i)$ consists of $1$'s;
\item $\begin{pmatrix}
B_{00}^i \ \  B_{01}^i\\
B_{10}^i  \ \  B_{11}^i
\end{pmatrix}$ is a reduced form whose Gross-Keating invariant consists of $0$'s or $1$'s;
\item the first integer of $\mathrm{GK}(B_{22}^i)$ is at least $2$.
\end{enumerate}
Here each block $B^i_{ij}$ can be the empty matrix.

  \end{lemma}
  \begin{proof}
We write a reduced form $B^i$ as $(b_{st})$, $1 \leq s,t \leq n$. 
Let $l$ be the number of    $0$'s in $\mathrm{GK}(L\cap 2^i  L^\sharp)$ and let $m$ be the number of    $1$'s in $\mathrm{GK}(L\cap 2^i  L^\sharp)$.
Consider an $\mathrm{GK}(L\cap 2^i  L^\sharp)$-admissible  involution $\sigma$ associated to the reduced form $B^i$.
Then using  the definition of a reduced form given in Definition \ref{def3.2} and the definition of an admissible involution given in Definition \ref{def3.1}, it is easy to see that  $\sigma$ and $b_{st}$ satisfy the following conditions:
\begin{enumerate}
\item[(i)] If $l$ ($m$, respectively) is even, then $\sigma (s)\neq s$ for  $1\leq s \leq l$ ($l+1\leq s \leq l+m$, respectively);
\item[(ii)] If $l$ is odd, then there is exactly one $s_0$ with $1\leq s_0 \leq l$ such that  either $\sigma(s_0)=s_0$ or $\sigma(s_0)> l+m$.
In this case, $\mathrm{ord}(b_{s_0s_0})=0$;
\item[(iii)] If $m$ is odd, then there is exactly one $s_1$ with $l+1\leq s_1 \leq l+m$ such that either  $\sigma(s_1)=s_1$ or $\sigma(s_1)> l+m$.
In this case, $\mathrm{ord}(b_{s_1s_1})=1$;
\item[(iv)] $\mathrm{ord}(2b_{st})\geq 1$ for any $1 \leq s\leq l$ and $l+1\leq t\leq l+m$;
\item[(v)] $\mathrm{ord}(b_{ss}) $ for any $s\geq l+m+1$ and $\mathrm{ord}(2b_{st})\geq 2$ for any $s,t\geq l+m+1$ with $s\neq t$.
\end{enumerate}

Let $B^i_{00}=(b_{st})$ with $1 \leq s,t \leq l$ and let  $B^i_{11}=(b_{st})$ with $l+1 \leq s,t \leq l+m$.
If $l=0$ ($m=0$, respectively), then we understand  $B^i_{00}$ ($B^i_{11}$, respectively) as the empty matrix.
Note that $B^i_{00}$ and $B^i_{11}$ determine the rest blocks of $B^i$. 
Using the definition of a reduced form given in Definition \ref{def3.2}, it is easy to see the followings:

The  statement (1) is obvious from the construction of $B^i_{00}$.
The  statement (2) follows from  (i) and (ii). 
The  statement (3)  follows from (i) and (iii).
The  statement  (4)  follows from (i)-(iv).
The  statement (5)  follows from (v).
  \end{proof}

We define  the nonnegative integer $m_i$ for each $i$ to be
$$m_i=\#\{t| a_t=1\}.$$
The reduced form $B^i$ has the following properties by Proposition \ref{propi-ii} and Remark \ref{rmk1}.(3), 
that are independent of the choices involved in its definition.
\begin{equation}\label{eqqq}
\left\{
  \begin{array}{l}
 \textit{If $L_i$ is of type $I$, then the rank of $B_{00}^i$ is $1$};\\
 \textit{If $L_i$ is of type $II$, then $B_{00}^i$ is empty};\\
 \textit{$B_{11}^i$ is an $m_i\times m_i$-matrix.}
    \end{array} \right.
\end{equation}
 
Let $q_{11}^i$ be the integral quadratic form represented by the symmetric matrix  $B_{11}^i$
and let $\bar{q}_{11}^i$ be the quadratic form $\frac{1}{2•}\cdot q_{11}^i$ modulo $2$, which is represented by the symmetric matrix $\frac{1}{2•}\cdot B^i_{11}$ modulo $2$. 
We claim the following lemma:
\begin{lemma}
The quadratic form $\bar{q}_{11}^i$ is  the same as the nonsingular quadratic form $\bar{q}_i$ lying over the quadratic space $\bar{V}_i$, given	in Section \ref{sec:5.1}. 
\end{lemma}
\begin{proof}
We first claim that
 the symmetric matrix $\tilde{B}^i=\begin{pmatrix}
4B_{00}^i&2B_{01}^i&2B_{02}^i\\
2B_{10}^i&B_{11}^i&B_{12}^i\\
2B_{20}^i&B_{21}^i&B_{22}^i
\end{pmatrix}$
represents the restriction of $\frac{1}{2^{i}}Q_L$ to  $B(L\cap 2^i  L^\sharp)$, 
 where $B(L\cap 2^i  L^\sharp)$ is defined as in (7) of the list of terminology introduced shortly following the beginning of Section \ref{sec:5.1}.
To prove this, we express $B^i=(b_{st})$, $1 \leq s,t \leq n$, where $B^i$ is as explained in Lemma \ref{lem48}. 
Since the symmetric matrix $B^i$ represents the integral quadratic form $\frac{1}{2^i}Q_L$ restricted to $L\cap 2^i  L^\sharp$, we have the following:
 \begin{equation}\label{eqn43}
 \frac{1}{2^i}Q_L|_{L\cap 2^i  L^\sharp}(x_1, \cdots, x_n)=
\sum_{s=1}^{n}b_{ss}x_s^2+2\sum_{1\leq s<t\leq n}b_{st}x_sx_t.  
\end{equation}

If $L_i$ is of type $II$ then $B^i_{00}$ is empty and the entry $b_{ss}$ for $1\leq s \leq n$ is contained in the ideal $(2)$ of $\mathfrak{o}$ by (\ref{eqqq}) and so  $B(L\cap 2^i  L^\sharp)=L\cap 2^i  L^\sharp$.
This verifies our claim.

If $L_i$ is of type $I$ then $b_{11}$ is a unit in $\mathfrak{o}$ and $b_{ss}$ for $2\leq s \leq n$ is contained in the ideal $(2)$ of $\mathfrak{o}$ by (\ref{eqqq}).
Therefore, we have the following description of the integral quadratic form $\frac{1}{2^{i}}Q_L$ restricted to $B(L\cap 2^i  L^\sharp)$:
\[
\frac{1}{2^{i}}Q_L|_{B(L\cap 2^i  L^\sharp)}(x_1, \cdots, x_n)=
4b_{11}x_1^2+\sum_{s=2}^{n}b_{ss}x_s^2+4\sum_{1< t\leq n}b_{1t}x_1x_t  +2\sum_{2\leq s<t\leq n}b_{st}x_sx_t.  
 \]
Here we replace $x_1$ of RHS of Equation (\ref{eqn43}) by $2x_1$.
Now it is clear that this quadratic form is represented by the symmetric matrix $\tilde{B}^i$.


Note that  any non-diagonal entry of $\tilde{B}^i$ multiplied by $2$ as well as  any diagonal entry of $\tilde{B}^i$, except entries of $B_{11}^i$, is divisible by $4$.
Since $\bar{q}_i$ is defined to be the quadratic form represented by the symmetric matrix  $\frac{1}{2•}\cdot \tilde{B}^i$ modulo $2$,
we have that $\bar{q}_i=\bar{q}^i_{11}$.
\end{proof}

The above lemma implies that   the quadratic form $\bar{q}_{11}^i$ is independent of the choice of a reduced form  $B^i$.
We obtain the following result:
\begin{lemma} 
\label{lem:5.4}
Recall that  $\mathrm{GK}(L\cap 2^i  L^\sharp)=(a_1, \cdots, a_n)$ and that $m_i=\#\{t  | a_t=1\}$.
Then \[m_i=\textit{dim $\bar{V}_i$}.\]
\end{lemma}

\begin{definition}\label{tegk}
We define  $\mathrm{GK}(L\cap 2^i  L^\sharp)^{\leq 1}$ (respectively $\mathrm{EGK}(L\cap 2^i  L^\sharp)^{\leq 1}$ ) to be the Gross-Keating invariant (respectively the extended Gross-Keating datum) of the reduced form 
$\begin{pmatrix}
B_{00}^i&B_{01}^i\\
B_{10}^i&B_{11}^i
\end{pmatrix}$, so that 
the only integers appearing in 
$\mathrm{GK}(L\cap 2^i  L^\sharp)^{\leq 1}$ are $0$ or $1$. 
\end{definition}
We remark that  truncated invariant $\mathrm{EGK}(L\cap 2^i  L^\sharp)^{\leq 1}$ is much simpler than $\mathrm{EGK}(L\cap 2^i  L^\sharp)$.

\begin{proposition}\label{prop:5.5}
We have a formula for $\mathrm{EGK}(L\cap 2^i  L^\sharp)^{\leq 1}$ as follows: (for the notion of $\mathrm{EGK}$, see Definition \ref{def:3.3}):
\begin{enumerate}
\item Assume that $m_i$ is even.  
Then we have
 \[
\mathrm{EGK}(L\cap 2^i  L^\sharp)^{\leq 1}=\left\{
  \begin{array}{l l}
 (1, m_i;0, 1;1, \zeta_2)   & \quad  \textit{if $L_i$ is of type $I$};\\
(m_i; 1;\zeta_1) & \quad  \textit{if $L_i$ is of type $II$}.
    \end{array} \right.
\]
Here, 
\[\textit{$\bar{q}_{11}^i=\bar{q}_{i}$ is split if and only if $\zeta_2=1$ (resp. $\zeta_1=1$) in the first case (resp. the second case).}\]
If $m_i=0$, then we understand the right hand side to be $(1;0;1)$ if $L_i$ is of type $I$
and zero if $L_i$ is of type $II$.

\item Assume that $m_i$ is odd. 
Then we have
 \[
\mathrm{EGK}(L\cap 2^i  L^\sharp)^{\leq 1}=\left\{
  \begin{array}{l l}
 (1, m_i;0, 1;1, 0)   & \quad  \textit{if $L_i$ is of type $I$};\\
(m_i;1;1) & \quad  \textit{if $L_i$ is of type $II$}.
    \end{array} \right.
\]
\end{enumerate}
\end{proposition}

Before proving the proposition, we state the following corollary, which is a direct consequence of the proposition.
\begin{corollary}
$\mathrm{EGK}(L\cap 2^i  L^\sharp)^{\leq 1}$ is independent of the choice of a reduced form $B^i$.
In other words, for given two reduced forms $B^i$ and $C^i$ for $L\cap 2^i  L^\sharp$,
$\mathrm{EGK}(L\cap 2^i  L^\sharp)^{\leq 1}$ associated to $B^i$ is the same as $\mathrm{EGK}(L\cap 2^i  L^\sharp)^{\leq 1}$ associated to $C^i$.
In addition, 
$\mathrm{EGK}(L\cap 2^i  L^\sharp)^{\leq 1}$ determines whether $\bar{q}_{11}^i$ is split or not.
\end{corollary}

\begin{proof}[Proof of Proposition \ref{prop:5.5}]
For (1), we first assume that $m_i$ is even and that $L_i$ is of type $I$.
The case with $m_i=0$ is easy and so we leave it as an exercise. 
Thus we also assume that $m_i$ is nonzero.
By Proposition 3.1 of \cite{IK1},
the quadratic lattice represented by the symmetric matrix $\begin{pmatrix}
B_{00}^i&B_{01}^i\\
B_{10}^i&B_{11}^i
\end{pmatrix}$ is unimodular of type $I^o$.
Thus, using Theorem 2.4 of \cite{Cho}, we have that  \[
\begin{pmatrix}
B_{00}^i&B_{01}^i\\
B_{10}^i&B_{11}^i
\end{pmatrix} \cong
(\epsilon)\perp \begin{pmatrix} 0&1\\ 1&0 \end{pmatrix}\perp \cdots 
\perp \begin{pmatrix} 0&1\\ 1&0 \end{pmatrix} \perp \begin{pmatrix} 2&1\\ 1&2a \end{pmatrix}
\]
 for $\epsilon\equiv 1\textit{ mod 2}$ and $a\in \mathfrak{o}$.
 Note that the right hand side is a reduced form with the involution $\sigma$ such that $\sigma(1)=1$ and $\sigma(i)=(i+1)$ for all even integer $i$.
Thus we can see that  
$\mathrm{EGK}(L\cap 2^i  L^\sharp)^{\leq 1}=(1, m_i;0, 1;1, \zeta_2)$.
To analyze the argument about $\zeta_2$,
we observe that $\bar{q}_{11}^i$ is split if and only if  
the equation $x^2+xy+\bar{a}y^2=0$ has a solution over $\frkk$, other than $(x,y)=(0,0)$.
Here $\bar{a}$ is the image of $a$ in $\frkk$.

On the other hand,  $\zeta_2=1$  if and only if  
$\begin{pmatrix}
B_{00}^i&B_{01}^i\\
B_{10}^i&B_{11}^i
\end{pmatrix} $ is split over $F$.
This is equivalent that 
$(\epsilon)\perp \begin{pmatrix} 2&1\\ 1&2a \end{pmatrix}$ is isotropic over $F$,  if and only if  $\begin{pmatrix} 2&1\\ 1&2a \end{pmatrix} \perp (4\epsilon)$ is isotropic over $F$.
The last condition is equivalent to the condition that $\begin{pmatrix} 2&1\\ 1&2a \end{pmatrix} \perp (4\epsilon)$ is isotropic over $\mathfrak{o}$, 
 if and only if  there is a solution $(x,y,z)$ (with $(x,y,z)\neq (0,0,0)$) over $\mathfrak{o}$ of $x^2+xy+ay^2+2\epsilon z^2=0$.

We claim that this is equivalent that the equation $x^2+xy+\bar{a}y^2=0$ has a nontrivial solution over $\frkk$, other than $(x,y)=(0,0)$.
Assume that $x^2+xy+ay^2+2\epsilon z^2=0$ has a nonzero solution over $\mathfrak{o}$ denoted by $(x_0, y_0, z_0) (\neq (0,0,0))$.
Since this polynomial is homogeneous, we may assume that at least one of $x_0, y_0, z_0$ is a unit in $\mathfrak{o}$.
If either $x_0$ or $y_0$ is a unit, then the equation $x^2+xy+\bar{a}y^2=0$ has a nontrivial solution over $\frkk$, by simply taking reduction modulo $\vpi$.
If both $x_0$ and $y_0$ are non-units and $z_0$ is a unit, we should have $x_0^2+x_0y_0+ay_0^2=-2\epsilon z_0^2$.
The exponential valuation of the left hand side is at least $2$, whereas that of the right hand side is $1$. This is a contradiction.

Conversely, if the equation $x^2+xy+\bar{a}y^2=0$ has a nontrivial solution over $\frkk$, then 
the equation $x^2+xy+ay^2+2\epsilon z^2=0$  also has a nontrivial solution over $\mathfrak{o}$, by the multivariable version of Hensel's lemma (cf. Theorems 2.1 and 3.8 of \cite{Con} or Proposition 5 in Section 2.3 of \cite{BLR}).
This verifies our claim about $\zeta_2$.



Secondly, we assume that $m_i$ is even and that $L_i$ is of type $II$ so that $B_{00}^i$ is empty.
Then as in the above case,
 Proposition 3.1 of \cite{IK1} yields that
the quadratic lattice represented by the symmetric matrix $B_{11}^i$ is unimodular of type $II$.
Thus using  Theorem 2.4 of \cite{Cho}, we have that 
\[
B_{11}^i \cong
 \begin{pmatrix} 0&1\\ 1&0 \end{pmatrix}\perp \cdots 
\perp \begin{pmatrix} 0&1\\ 1&0 \end{pmatrix} \perp \begin{pmatrix} 2a&1\\ 1&2b \end{pmatrix}
\]
 for  $a, b\in \mathfrak{o}$, so that 
$\mathrm{EGK}(L\cap 2^i  L^\sharp)^{\leq 1}=(m_i; 1;\zeta_1)$ and  
 $F(\sqrt{D_{B_{11}^i}})=F(\sqrt{1-4ab})$.
 Here, $D_{B_{11}^i}$ is defined in the paragraph before Definition \ref{def:0.3}.
 
The field extension $F(\sqrt{1-4ab})/F$ is the splitting field of the equation $x^2-x+ab=0$.
Thus  $F(\sqrt{1-4ab})/F$ is either trivial or nontrivial unramified.
In addition, $F(\sqrt{1-4ab})/F$ is trivial  if and only if  $\zeta_1=1$  if and only if 
the equation $x^2-x+ab=0$ has a solution over $\mathfrak{o}$.
Hensel's lemma yields that this is equivalent to the existence of a solution of the equation $x^2+x+\bar{a}\bar{b}=0$ over $\frkk$.

We claim that this is equivalent to the condition that  $\bar{q}_{11}^i$ is split,  if and only if  the equation $\bar{a}x^2+xy+\bar{b}y^2=0$ has a nontrivial solution over $\frkk$.
If $\bar{a}=0$, then both equations have nontrivial solutions.
If $\bar{a}\neq 0$, then the equation $\bar{a}x^2+xy+\bar{b}y^2=0$ is equivalent to $(\bar{a}x)^2+(\bar{a}x)y+\bar{a}\bar{b}y^2=0$.
Then the existence of a nontrivial solution of  $(\bar{a}x)^2+(\bar{a}x)y+\bar{a}\bar{b}y^2=0$ is equivalent to the existence of a solution
of  $x^2+x+\bar{a}\bar{b}=0$ over $\frkk$.
This verifies our claim about $\zeta_1$.

For (2),  assume that $m_i$ is odd.
As in the above case,
 Proposition 3.1 of \cite{IK1} yields that
the quadratic lattice represented by the symmetric matrix $\begin{pmatrix}
B_{00}^i&B_{01}^i\\
B_{10}^i&B_{11}^i
\end{pmatrix}$ has two Jordan components, say $M_0\oplus M_1$, where $\bfs(M_j)=\frkp^j$ with $j=0 \textit{ or }1$.
Note that $M_0$ could be empty depending on the type and the rank of $L_i$, which will be explained below.
If $L_i$ is of type $I$, then  the rank of $M_0$ (resp. $M_1$) is $m_i$ (resp. $1$).  
Thus using  Theorem 2.4 of \cite{Cho}, we have that 
\[
\begin{pmatrix}
B_{00}^i&B_{01}^i\\
B_{10}^i&B_{11}^i
\end{pmatrix} \cong
(\epsilon)\perp \begin{pmatrix} 0&1\\ 1&0 \end{pmatrix}\perp \cdots 
\perp \begin{pmatrix} 0&1\\ 1&0 \end{pmatrix} \perp \begin{pmatrix} 2&1\\ 1&2a \end{pmatrix}\perp (2\epsilon')
\]
 for $\epsilon, \epsilon' \in \mathfrak{o}^{\times}$ and $a\in \mathfrak{o}$.
 If $m_i=1$, then we understand the right hand side to be $(\epsilon)\perp (2\epsilon')$.
 The right hand side is a reduced form of GK-type $(\underline{a}, \sigma)$, where
 $\underline{a}=(0, 1, \cdots, 1)$ and $\sigma$ exchanges   $t$ and $t+1$ for all even integer $t$ with $2\leq t\leq m_i-1$, 
 and fixes $1$ and $m_i+1$. 
Using this, we can easily see that
$\mathrm{EGK}(L\cap 2^i  L^\sharp)^{\leq 1}=(1, m_i;0, 1;1, 0)$.

If $L_i$ is of type $II$, then $B_{00}^i$ is empty and thus 
 the rank of $M_0$ (resp. $M_1$) is $m_i-1$ (resp. $1$).  
 The case with $m_i=1$ is easy and so we leave it as an exercise. 
 Thus we assume that $m_i>1$.
Using  Theorem 2.4 of \cite{Cho},  we have that 
\[
B_{11}^i \cong
\begin{pmatrix} 0&1\\ 1&0 \end{pmatrix}\perp \cdots 
\perp \begin{pmatrix} 0&1\\ 1&0 \end{pmatrix} \perp \begin{pmatrix} 2a&1\\ 1&2b \end{pmatrix}\perp (2\epsilon')
\]
 for $\epsilon'x` \in \mathfrak{o}^{\times}$  and $a, b\in \mathfrak{o}$.
On the other hand, the quadratic lattice represented by 
$\begin{pmatrix} 2a&1\\ 1&2b \end{pmatrix}\perp (2\epsilon')$ is always isotropic, 
as can be proved using
Hensel's lemma.
Thus  we have that
$\mathrm{EGK}(L\cap 2^i  L^\sharp)^{\leq 1}=(m_i;1;1)$.
\end{proof}


\subsection{Final result}\label{secfr}
We now state our main theorem of this section.

\begin{theorem}\label{thm-ldgk}
The local density $\beta(L)$ is completely determined by the collection consisting of  $\mathrm{GK}(L\oplus -L)$ 
together with the $\mathrm{EGK}(L\cap 2^i  L^\sharp)^{\leq 1}$ as $i$ runs over all integers  such that $L_i$ is nonzero, where $L=\oplus_i L_i$ is a Jordan splitting.
In other words,  given any two quadratic lattices $L, M$ satisfying 
\[
\left\{
  \begin{array}{l}
 \textit{$\mathrm{GK}(L\oplus -L)=\mathrm{GK}(M\oplus -M)$};\\
 \textit{$\mathrm{EGK}(L\cap 2^i  L^\sharp)^{\leq 1}=\mathrm{EGK}(M\cap 2^i  M^\sharp)^{\leq 1}$ for all $i$},
    \end{array} \right.
\]
we have that 
\[\beta(L)=\beta(M).\]
\end{theorem}

\begin{proof}
Firstly, 
by Proposition \ref{prop:5.5}, the parity type of $L_i$ is $I$ if and only if $0\in \mathrm{GK}(L\cap 2^i  L^\sharp)^{\leq 1}$.
Thus the collection of $\{\mathrm{GK}(L\cap 2^i  L^\sharp)^{\leq 1}|i\in \mathbb{Z}, L_i\neq 0\}$  determines the integer $\beta$ in Theorem \ref{thmcho412} and the integers 
$t$ and $b$ in  Theorem \ref{thmcho52}.
 Corollary  \ref{cortypelii} implies that $\mathrm{GK}(L\oplus -L)$ and the collection of $\{\mathrm{GK}(L\cap 2^i  L^\sharp)^{\leq 1}|i\in \mathbb{Z}, L_i\neq 0\}$ determine $\mathcal{A}_i$'s and $\mathcal{B}_i$'s.
Thus by Corollary \ref{cortypeli}, the rank of each $L_i$ (denoted by $n_i$) is determined by  $\mathrm{GK}(L\oplus -L)$ together with the collection of $\{\mathrm{GK}(L\cap 2^i  L^\sharp)^{\leq 1}|i\in \mathbb{Z}, L_i\neq 0\}$.
Therefore, these two invariants determine $N$ as well as $c$ in  Theorem \ref{thmcho52}.

Secondly,  the quadratic space $\bar{V_i}$ (and thus $\#\mathrm{O}(\bar{V_i}, \bar{q_i})^{\mathrm{red}}$) is determined by 
$\mathrm{EGK}(L\cap 2^i  L^\sharp)^{\leq 1}$ as we see using Lemma \ref{lem:5.4} (which determines the dimension of $\bar{V_i}$) 
and Proposition \ref{prop:5.5} (which determines whether or not $\bar{V_i}$ is split).
Thus it suffices to show that the integer $\alpha$ in Theorem \ref{thmcho412} is determined by $\mathrm{GK}(L\oplus -L)$ and 
the collection of $\{\mathrm{GK}(L\cap 2^i  L^\sharp)^{\leq 1}|i\in \mathbb{Z}, L_i\neq 0\}$.

These two invariants determine each  parity type and the rank of $L_i$ as explained above. 
Assume that $L_i$ is free of type $I^e$. 
Then  $L_i$ is of type $I^e_1$ if and only if the dimension of $\bar{V}_i$ is odd.
The latter is determined by $\mathrm{GK}(L\cap 2^i  L^\sharp)^{\leq 1}$ (cf. Lemma \ref{lem:5.4}).
This completes the proof.
\end{proof}

\begin{remark}\label{rmk3}
As in Remark \ref{rmk2}, 
the sequence   $\{\mathrm{EGK}(L\cap 2^i  L^\sharp)^{\leq 1}\}_{L_i\neq 0}$  is not enough to describe the local density $\beta(L)$.
As an  example, Let $L$ be the quadratic lattice represented by the symmetric matrix 
$\begin{pmatrix} \begin{pmatrix} 0&1\\1&0  \end{pmatrix}&0\\0&\begin{pmatrix} 1&0\\0&3  \end{pmatrix}  \end{pmatrix}$
and let $M$ be the quadratic lattice represented by the symmetric matrix 
$\begin{pmatrix} \begin{pmatrix} 0&1\\1&0  \end{pmatrix}&0\\0&1  \end{pmatrix}$.
Then $L$ is unimodular of type $I^e$ and $M$ is unimodular of type $I^o$ 
 so that they have the different local densities.
But we can easily show that 
\[\EGK(L)^{\leq 1}=\EGK(M)^{\leq 1} (=\EGK(M))=(1,2;0,1;1,1).
\]

\end{remark}

\begin{remark}\label{podd}
In this remark, we  assume that $F$ is a finite field extension of $\mathbb{Q}_p$ with $p>2$.
Then a quadratic lattice $(L, Q_L)$ is always diagonalizable. 
That is, there is a basis of $L$ such that with respect to this basis the  symmetric matrix of the quadratic lattice  $(L, Q_L)$  is $(b_1)\perp \cdots \perp (b_n)$, where $\mathrm{ord}(b_i)\leq \mathrm{ord}(b_j)$ if $i  \leq j$.
If we define $a_i$  to be $\mathrm{ord}(b_i)$, then 
Remark 1.1 of \cite{IK1} says that $\mathrm{GK}(L)=(a_1, \cdots, a_n)$.
Thus $\mathrm{GK}(L)$ determines $n_i$ by the relation $n_i=\#\{a_j|a_j=i\}$. 
Here, $n_i$ is 
 the rank of $L_i$, where $L=\oplus L_i$ is a Jordan splitting such that  $\bfs(L_i)=\frkp^i$.
 
 On the other hand,
one of the principal results of \cite{GY} implies that the local density $\beta(L)$ is completely determined by all the $n_i$.
More precisely, by Theorem 7.3 of loc. cit., $\beta(L)$ is determined by all the $n_i$ and the $\#G_i(\frkk)$.
For the definition of $G_i$, see Section 6.2 of loc. cit. 
The algebraic group $G_i$ defined over $\frkk$ is described  in Proposition 6.2.3 of loc. cit. 
It is basically an orthogonal group associated to a nondegenerate quadratic space of dimension $n_i$ over $\frkk$.
If $n_i$ is odd, then $G_i$ is split. 
If $n_i$ is even, then it could be either split or non-split.
Nonetheless,  $\#G_i(\frkk)$ only depends on the field $\frkk$ and the dimension of the quadratic space, which is $n_i$, as described in page 818 of \cite{Cho}.

In conclusion, the local density $\beta(L)$ is completely determined by $\mathrm{GK}(L)$.
\end{remark}

\appendix
\section{The local density of a binary quadratic form}
\label{App:Appendix}

\centerline{Tamotsu IKEDA} 
\centerline{Graduate school of mathematics, Kyoto University,} 
\centerline{Kitashirakawa, Kyoto, 606-8502, Japan}
\centerline{ikeda@math.kyoto-u.ac.jp}
\medskip

\centerline{Hidenori KATSURADA}
\centerline{Muroran Institute of Technology}
\centerline{27-1 Mizumoto, Muroran, 050-8585, Japan}
\centerline{hidenori@mmm.muroran-it.ac.jp}
\bigskip

In this appendix, we calculate the local density of a binary form over a dyadic field $F$ which may not be an unramified extension of $\QQ_2$.
We also calculate the Gross-Keating invariant $\GK(L\perp -L)$ and the truncated EGK invariant $\EGK(L\cap \vpi^i L^\sharp)^{\leq 1}$ for a binary quadratic lattice $L$.
The local density formula is given in Proposition \ref{prop:A6}.
We also show that the local density is not determined by $\GK(L\perp -L)$ and $\EGK(L\cap \vpi^i L^\sharp)^{\leq 1}$, if we drop the assumption that $F/\QQ_2$ is unramified (See Example \ref{ex:A1}).

Let $\frko$ be the ring of integers of $F$, $\frkp$ the maximal ideal of $\frko$,  $\vpi$ a prime element of $F$, and $q$ the order of the residue field.
The ramification index of $F/\QQ_2$ is denoted by $e$.

Let $(L, Q)$ and $(L_1, Q_1)$ be quadratic lattices of rank $n$ over $\frko$.
We say that $(L, Q)$ and $(L_1, Q_1)$ are weakly equivalent if there exist an isomorphism $\iot:L\rightarrow L_1$ and a unit $u\in \frko^\times$ such that $u Q_1(\iot(x))=Q(x)$ for any $x\in L$.
Similarly, we say that $B, B_1\in \calh_n(\frko)$ are weakly equivalent if there exist a unimodular matrix $U\in\GL_n(\frko)$ and a unit $u \in\frko^\times$ such that $u B_1=B[U]$.
If $B$ and $B_1$ are weakly equivalent, then $\GK(B)=\GK(B_1)$.
Recall that a half-integral symmetric matrix $B\in\calh_n(\frko)$ is primitive if $\vpi^{-1}B\notin \calh_n(\frko)$.
Put $\GK(B)=(a_1, a_2, \ldots, a_n)$.
Then $B$ is primitive if and only if $a_1=0$.

Let $E/F$ be a  semi-simple quadratic algebra.
This means that $E$ is a quadratic extension of $F$ or $E=F\times F$.
The non-trivial automorphism of $E/F$ is denoted by $x\mapsto \bar x$.
Note that if  $E=F\times F$, we have $\overline{(x_1, x_2)}=(x_2, x_1)$.
Let $\frko_E$ be the maximal order of $E$.
In the case $E=F\times F$, $\frko_E=\frko\times \frko$.
The discriminant ideal of $E/F$ is denoted by $\frkD_E$.
When $E=F\times F$, we understand $\frkD_E=\frko$.
Put $d=\ord(\frkD_E)$ and
\[
\xi=
\begin{cases}
1 & \text{ if $E=F\times F$, } \\
-1 & \text{ if $E/F$ is unramified quadratic extension, } \\
0 & \text{ if $E/F$ is ramified quadratic extension. }
\end{cases}
\]
We say that $E/F$ is unramified, if $d=0$.
Note that $d\in\{2g\,|\, g\in\ZZ, 0\leq g \leq e\}\cup\{2e+1\}$.
The order $\frko_{E, f}$ of conductor $f$ for $E/F$ is defined by $\frko_{E, f}=\frko+\frkp^f\frko_E$.
Any open $\frko$-subring of $\frko_E$ is of the form $\frko_{E, f}$ for some non-negative integer $f$.

\begin{proposition}[\cite{IK1}, Proposition 2.1] 
\label{prop:6.1}
Let $B\in\calhnd_2(\frko)$ be a primitive half-integral symmetric matrix of size $2$ and $(L, Q)$ its associated quadratic lattice.
Put $E=F(\sqrt{D_B})/F$.
When $D_B\in F^{\times 2}$, we understand $E=F\times F$.
Put $f=(\ord(D_B)-\ord(\mathfrak{D}_E))/2$.
Then $f$ is a non-negative integer and $(L, Q)$ is weakly equivalent to $(\frko_{E, f}, \caln)$, where $\caln$ is the norm form for $E/F$.
\end{proposition}

\begin{proposition}[\cite{IK1}, Proposition 2.2] 
\label{prop:6.2}
The Gross-Keating invariant of the binary quadratic form $(L, Q)=(\frko_{E,f}, \caln)$ is given by
\[
\begin{cases}
(0, 2f) & \text{ if $E/F$ is unramified,}
\\
(0, 2f+1) & \text{ if $E/F$ is ramified.}
\end{cases}
\]
\end{proposition}

The following lemma is well-known.
\begin{lemma} 
\label{lem:6.3}
We have
\[
[\frko^\times:\frko^{\times 2}(1+\frkp^f)]=
\begin{cases}
q^{\left[\frac f2\right]} & \text{ if $0<f\leq 2e$,} \\
2q^e & \text{ if $f> 2e$.} 
\end{cases}
\]
\end{lemma}

Choose $\ome\in \frko_E$ such that $\frko_E=\frko+\frko\ome$.
If $E/F$ is unramified, then $\ord_E(\ome)=0$.
It $E/F$ is ramified, then we may assume $\ome$ is a prime element of $E$.
Put $h=\ord_E(\ome)$.

We fix an $\frko$-module isomorphism $\frko^2\simeq \frko_{E, f}$ by $(x, y)\mapsto x+\vpi^f \ome y$.
By this isomorphism, we identify $\frko^2$ and $\frko_{E,f}$.
We consider a quadratic form $Q(x, y)$ by
\[
Q(x, y)=\caln(x+\vpi^f \ome y)=x^2+\vpi^f \mathrm{tr}(\ome) xy +\vpi^{2f}\caln(\ome)y^2.
\]
An $\frko$-endomorphism of $\frko_{E,f}$ is expressed as 
\[
U_{\alp, \bet}(x+ \vpi^f \ome y) = \alp x + \bet \vpi^f \ome y
\]
for some $\alp\in \frko_{E, f}$ and $\bet\in (\vpi^f\ome)^{-1}\frko_{E, f}$.
Note that $U_{\alp, \alp}\circ U_{\bet, \gam}=U_{\alp\bet, \alp\gam}$.
Note also that
\[
Q(U_{\alp, \bet}(x,y))=
\caln(\alp)x^2+\vpi^f \mathrm{tr}(\bar\alp\ome\bet) xy +\vpi^{2f}\caln(\ome\bet)y^2.
\]
We shall determine when $Q\circ U_{\alp, \bet}\equiv Q$ mod $\frkp^N$, where $N$ is a sufficiently large integer.
Put 
\[
V_N=\{\alp\in \frko_{E,f}\,|\, \caln(\alp)\equiv 1 \text{ mod } \frkp^N\}.
\]
Then $V_N\subset \frko_{E,f}^\times$.
Clearly, if $Q\circ U_{\alp, \bet}\equiv Q$ mod $\frkp^N$, then $\alp\in V_N$.
Replacing $U_{\alp, \bet}$ by $U_{\alp, \alp}^{-1}\circ U_{\alp, \bet}$, we may assume $\alp=1$.
Then $\bet$ belongs to the set
\[
W_N=\left\{\bet\in (\vpi^f\ome)^{-1}\frko_{E, f}\ \vrule\ \begin{matrix} \mathrm{tr}(\ome\bet)\equiv \mathrm{tr}(\ome) \text{ mod } \frkp^{N-f} \\ \caln(\bet)\equiv 1 \text{ mod } \frkp^{N-2f-h}\ \  \end{matrix}\right\}.
\]
Thus we have
\begin{align*}
\{U_{\alp, \bet}\,|\,\text{$Q\circ U_{\alp, \bet}\equiv Q$ mod $\frkp^N$}\}
=\{U_{\alp, \alp}\circ U_{1, \bet}\,|\, \alp\in V_N, \ \bet\in W_N\},
\end{align*}
As we have assumed that $N$ is sufficiently large, we have $W_N\subset \frko_{E,f}^\times$.
Then  the local density for $(\frko_{E,f}, Q)$ is 
\[
\frac 12 q^{3N}
\frac{\mathrm{Vol}(V_N)} {\mathrm{Vol}(\frko_{E,f})}
\frac{\mathrm{Vol}(W_N)} {\mathrm{Vol}((\vpi^f \ome)^{-1}\frko_{E,f})}.
\]
\begin{lemma} 
\label{lem:6.4}
\[
\frac{\mathrm{Vol}(W_N)} {\mathrm{Vol}((\vpi^f \ome)^{-1}\frko_{E,f})}
=2 q^{-2N+2f+d}.
\]
\end{lemma}
\begin{proof}
For $\bet\in W_N$, we have
\begin{align*}
(\bar \bet-1)(\ome \bet-\bar\ome)&\equiv \ome\caln(\bet)-\mathrm{tr}(\ome \bet
)+\bar \ome \\
&\equiv \ome-\mathrm{tr}(\ome)+\bar \ome \\
&\equiv 0 \qquad \qquad\qquad\text{ mod } \vpi^{N-2f}\frko_E.
\end{align*}
It follows that
\[
W_N\subset (1+\vpi^{N-2f-d}\frko_E)\cup \left(\frac{\bar\ome}{\ome}+\vpi^{N-2f-d}\frko_E\right).
\]
Put $W'_N=W_N\cap (1+\vpi^{N-2f-d}\frko_E)$.
Then we have 
\[
W_N=W'_N\cup \frac{\bar\ome}{\ome}\overline{W'_N}.
\]
Note that $\ord_E(1-\frac{\bar\ome}{\ome})=\ord_E(\ome^{-1}(\ome-\bar\ome))=d-h$, and so we have $W'_N\cap \frac{\bar\ome}{\ome}\overline{W'_N}=\emptyset$, since $N$ is sufficiently large.
Hence we have $\mathrm{Vol}(W_N)=2\mathrm{Vol}(W'_N)$.
Note that
\[
W'_N=
\left\{
1+\vpi^{N-2f-d}\gam\in 1+\vpi^{N-2f-d}\frko_E
\ \vrule\ \begin{matrix} \mathrm{tr}(\ome\gam)\equiv 0 \text{ mod } \frkp^{f+d} \\ \mathrm{tr}(\gam)\equiv 0 \text{ mod } \frkp^{d-h}\ \  \end{matrix}\right\}.
\]
Observe that if $\gam=x+\bar\ome y$, \ $x, y\in\frko$, then
\[
\begin{pmatrix}
\mathrm{tr}(\gam) \\
\mathrm{tr}(\ome\gam)
\end{pmatrix}
=
\begin{pmatrix}
2 & \mathrm{tr}(\ome) \\ \mathrm{tr}(\ome) & 2\caln(\ome)
\end{pmatrix}
\begin{pmatrix}
x \\ y
\end{pmatrix}.
\]
Since $\ord\left(\det \begin{pmatrix}
2 & \mathrm{tr}(\ome) \\ \mathrm{tr}(\ome) & 2\caln(\ome)
\end{pmatrix}\right)=\ord(\ome-\bar\ome)^2=d$, we have
\[
\frac{\mathrm{Vol}(W'_N)} {\mathrm{Vol}(1+\vpi^{N-2f-d}\frko_E)}
=
q^{-f-d+h}.
\]
Note also that
\[
\mathrm{Vol}((\vpi^f \ome)^{-1}\frko_{E,f})=q^{2f+h}\mathrm{Vol}(\frko_{E,f})=
q^{f+h}\mathrm{Vol}(\frko_{E}).
\]
Hence we have
\begin{align*}
\frac{\mathrm{Vol}(W_N)} {\mathrm{Vol}((\vpi^f \ome)^{-1}\frko_{E,f})}
&=2 \frac{\mathrm{Vol}(W'_N)} 
{\mathrm{Vol}(1+\vpi^{N-2f-d}\frko_E)}
\frac
{\mathrm{Vol}(1+\vpi^{N-2f-d}\frko_E)}
{\mathrm{Vol}((\vpi^f \ome)^{-1}\frko_{E,f})}
\\
&=
2 q^{-f-d+h}\cdot q^{-2N+3f+2d-h}
\\
&=2 q^{-2N+2f+d}.
\end{align*}
\end{proof}

\begin{lemma} 
\label{lem:6.5}
\begin{itemize}
\item[(1)]
If $E/F$ is unramified, then
\[
\frac{\mathrm{Vol}(V_N)} {\mathrm{Vol}(\frko_{E,f})}
=
\begin{cases}
q^{-N}(1-\xi q^{-1}) & \text{ if $f=0$,} 
\\ \noalign{\vskip 3pt}
q^{-N+[f/2]} & \text{ if $0 < f \leq 2e$,} 
\\ \noalign{\vskip 2pt}
2q^{-N+e} & \text{ if $f>2e$.}
\end{cases}
\]
\item[(2)] If $E/F$ is ramified, then
\[
\frac{\mathrm{Vol}(V_N)} {\mathrm{Vol}(\frko_{E,f})}
=
\begin{cases}
2q^{-N+f} & \text{ if $0\leq f < \left[\frac{d+1}2\right]$,} 
\\ \noalign{\vskip 5pt}
q^{-N+[\frac f2+\frac d4]} & \text{ if $\left[\frac {d+1}2\right] \leq f \leq 2e-\left[ \frac d2\right]$,} 
\\ \noalign{\vskip 5pt}
2q^{-N+e} & \text{ if $f> 2e-\left[\frac d2\right]$.}
\end{cases}
\]
\end{itemize}
\end{lemma}
\begin{proof}
We normalize the Haar measure of $E$ and $F$ by $\mathrm{Vol}(\frko_E)=\mathrm{Vol}(\frko)=1$.
Since $N$ is sufficiently large, $\caln(\frko_{E,f}^\times)\supset 1+\frkp^N$.
Then we have
\[
[\frko_{E,f}^\times :V_N]=[\caln(\frko_{E,f}^\times): 1+\frkp^N].
\]
We have
\begin{align*}
\frac{\mathrm{Vol}(V_N)} {\mathrm{Vol}(\frko_{E,f})}
=&
\frac{\mathrm{Vol}(\frko_{E,f}^\times)} {\mathrm{Vol}(\frko_{E,f})}
\frac{\mathrm{Vol}(1+\frkp^N)} {\mathrm{Vol}(\caln(\frko^\times_{E,f}))}
\\
=&
q^{-N+f}
\frac{\mathrm{Vol}(\frko_{E,f}^\times)}{\mathrm{Vol}(\caln(\frko^\times_{E,f}))}.
\end{align*}
It is easily seen that
\[
\mathrm{Vol}(\frko_{E,f}^\times)
=
\begin{cases}
(1-q^{-1})(1-\xi q^{-1}) & \text{ if $f=0$, } \\

q^{-f} (1-q^{-1}) & \text{ if $f>0$. } 
\end{cases}
\]
Thus it is enough to calculate $\mathrm{Vol}(\caln(\frko_{E,f}^\times))$.
If $f=0$, then 
\[
[\frko^\times :\caln(\frko_E^\times)]=
\begin{cases}
1 & \text{ if $E/F$ is unramified, } \\
2 & \text{ if $E/F$ is ramified.}
\end{cases}
\]
This settles the case $f=0$.
Suppose that $f>0$.
Then $\frko_{E,f}^\times=\frko^\times (1+\vpi^f \frko_E)$ and so
\[
\caln(\frko_{E,f}^\times)=\frko^{\times 2} \cdot\caln(1+\vpi^f \frko_E).
\]
If $E/F$ is unramified, then $\caln(1+\vpi^f\frko_E)=1+\frkp^f$.
By Lemma \ref{lem:6.3}, we have
\[
\mathrm{Vol}
(\caln(\frko_{E,f}^\times))
=
\begin{cases}
(1-q^{-1}) q^{-\left[\frac f2\right]} &\text{ if $0< f\leq 2e$,} \\
\frac 12 (1-q^{-1}) q^{-e} &\text{ if $f> 2e$,} 
\end{cases}
\]
and so
\[
\frac{\mathrm{Vol}(V_N)} {\mathrm{Vol}(\frko_{E,f})}
=
\begin{cases}
q^{-N+[f/2]} & \text{ if $0 < f \leq 2e$,} 
\\ \noalign{\vskip 2pt}
2q^{-N+e} & \text{ if $f>2e$.}
\end{cases}
\]
Now suppose $F/F$ is ramified.
By Serre \cite{serre}, p.85, Corollary 3, we have 
\[
\caln(1+\vpi^f \frko_E)=1+\frkp^{f+\left[\frac d2\right]}
\]
for $f\geq \left[ \frac{d+1} 2\right]$.
It follows that
\[
\caln(\frko_{E,f}^\times)=\frko^{\times 2}(1+\frkp^{f+\left[\frac d2\right]})
\]
for $f\geq \left[ \frac{d+1} 2\right]$.
By Lemma \ref{lem:6.3}, we have
\[
\mathrm{Vol}(\caln(\frko_{E,f}^\times))
=
\begin{cases}
(1-q^{-1})q^{-[\frac f2+\frac d4]} & \text{ if $\left[\frac {d+1}2\right] \leq f \leq 2e-\left[ \frac d2\right]$,} 
\\ \noalign{\vskip 5pt}
\frac 12 (1-q^{-1})q^{-e} & \text{ if $f> 2e-\left[\frac d2\right]$,}\end{cases}
\]
and so
\[
\frac{\mathrm{Vol}(V_N)} {\mathrm{Vol}(\frko_{E,f})}
=
\begin{cases}
q^{-N+[\frac f2+\frac d4]} & \text{ if $\left[\frac {d+1}2\right] \leq f \leq 2e-\left[ \frac d2\right]$,} 
\\ \noalign{\vskip 5pt}
2q^{-N+e} & \text{ if $f> 2e-\left[\frac d2\right]$.}
\end{cases}
\]
Finally, suppose that $0<f<\left[ \frac{d+1} 2\right]$.
In this case, by Shimura \cite{shimura}, Lemma 21.13 (v), we have
\[
\caln(1+\vpi^f \frko_E)=(1+\frkp^{2f})\cap \caln(\frko_E^\times).
\]
Since $1+\frkp^{d-1}\not\subset\caln(\frko_E^\times)$, we have
$\caln(1+\vpi^f \frko_E)\subsetneqq 1+\frkp^{2f}$.
Hence
\[
\mathrm{Vol}(\caln(1+\vpi^f\frko_E))=\frac 12 \mathrm{Vol}(1+\frkp^{2f})
=\frac 12 q^{-2f}.
\]
On the other hand, we have
\[
\caln(1+\vpi^f \frko_E)\cap \frko^{\times 2}=(1+\frkp^{2f})\cap \caln(\frko_E^\times)\cap \frko^{\times 2}=(1+\frkp^{2f})\cap \frko^{\times 2}.
\]
Hence
\begin{align*}
[\frko^{\times 2}:\caln(1+\vpi^f \frko_E)\cap \frko^{\times 2}]
=&
[\frko^{\times 2}:(1+\frkp^{2f})\cap \frko^{\times 2}] \\
=&
[\frko^{\times 2}(1+\frkp^{2f}): 1+\frkp^{2f}] \\
=&
\frac{[\frko^\times: 1+\frkp^{2f}]}
{[\frko^\times:\frko^{\times 2}(1+\frkp^{2f})]}
\\
=& q^f(1-q^{-1}).
\end{align*}
It follows that
\begin{align*}
\mathrm{Vol}(\caln(\frko_{E,f}^\times))=&
\mathrm{Vol}(\caln(1+\vpi^f\frko_E))
[\frko^{\times 2}\cdot\caln(1+\vpi^f \frko_E):\caln(1+\vpi^f \frko_E)]
\\
=&\frac 12 q^{-2f}[\frko^{\times 2}:\caln(1+\vpi^f \frko_E)\cap \frko^{\times 2}] \\
=&\frac 12 q^{-f}(1-q^{-1}).
\end{align*}
Hence we have
\[
\frac{\mathrm{Vol}(V_N)} {\mathrm{Vol}(\frko_{E,f})}
=
2q^{-N+f}  
\]
in this case.
This proves the lemma.
\end{proof}
By Lemma \ref{lem:6.4} and Lemma \ref{lem:6.5}, we obtain the following formula.

\begin{proposition} 
\label{prop:A6}
\begin{itemize}
\item[(1)] 
Assume that $E$ is unramified.
Then the local density of $(L, Q)=(\frko_{E,f}, \caln)$ is given by
\[
\beta(L)=
\begin{cases}
1-\xi q^{-1} & \text{ if $f=0$,} 
\\ \noalign{\vskip 3pt}
q^{[f/2]+2f} & \text{ if $0 < f \leq 2e$,} 
\\ \noalign{\vskip 2pt}
2q^{e+2f} & \text{ if $f>2e$.}
\end{cases}
\]
\item[(2)] 
Assume that $E$ is ramified and that $d=2g\leq 2e$.
Then the local density of $(L, Q)=(\frko_{E,f}, \caln)$ is given by
\[
\beta(L)=
\begin{cases}
2q^{3f+2g} & \text{ if $0\leq f < g$,} 
\\ \noalign{\vskip 5pt}
q^{[\frac f2+\frac g2]+2f+2g} & \text{ if $g \leq f \leq 2e-g$,} 
\\ \noalign{\vskip 5pt}
2q^{2f+e+2g} & \text{ if $f> 2e-g$.}
\end{cases}
\]
\item[(3)] 
Assume that $E$ is ramified and that $d=2e+1$.
Then the local density of $(L, Q)=(\frko_{E,f}, \caln)$ is given by
\[
\beta(L)=
\begin{cases}
2q^{3f+2e+1} & \text{ if $0\leq f < e+1$,} 
\\ \noalign{\vskip 5pt}
2q^{2f+3e+1} & \text{ if $f\geq e+1$.}
\end{cases}
\]
\end{itemize}
\end{proposition}

Next, we calculate $\GK(L\oplus -L)$.
\begin{proposition} 
\label{prop:A7}
Suppose that $(L, Q)=(\frko_{E,f}, \caln)$.
\begin{itemize}
\item[(1)] 
Assume that $E$ is unramified.
Then $L\oplus -L$ is equivalent to a reduced form of GK type $(\ua, \sig)$, where 
\[
(\ua, \sig)=
\begin{cases}
((0,0,0,0), \, (12)(34)) & \text{ if $f=0$,} 
\\ \noalign{\vskip 3pt}
((0,f,f,2f),\,  (14)(23)) & \text{ if $0 < f < 2e$,} 
\\ \noalign{\vskip 2pt}
((0, 2e, 2f-2e, 2f),\,  (12)(34)) & \text{ if $f\geq 2e$.}
\end{cases}
\]
\item[(2)] 
Assume that $E$ is ramified and that $d\leq 2e$.
Put $g=d/2$.
Then $L\oplus -L$ is equivalent to a reduced form of GK type $(\ua, \sig)$, where 
\[
(\ua, \sig)=
\begin{cases}
((0, 2f+1, 2g-1, 2g+2f),\, (14)(23)) & \text{ if $0\leq f \leq g-1$,} 
\\ \noalign{\vskip 5pt}
((0, g+f, g+f, 2g+2f),\, (14)(23)) & \text{ if $g\leq f < 2e-g$,} 
\\ \noalign{\vskip 5pt}
((0, 2e, 2g+2f-2e, 2g+2f),\, (12)(34)) & \text{ if $ f \geq 2e-g$.} 
\end{cases}
\]
\item[(3)] 
Assume that $E$ is ramified and that $d=2e+1$.
Then $L\oplus -L$ is equivalent to a reduced form of GK type $(\ua, \sig)$, where 
\[
(\ua, \sig)=
\begin{cases}
((0, 2f+1, 2e, 2e+2f+1),\, (13)(24)) & \text{ if $0\leq f < e$,} 
\\ \noalign{\vskip 5pt}
((0, 2e, 2f+1, 2e+2f+1),\, (12)(34)) & \text{ if $ f \geq e$.} 
\end{cases}
\]

\end{itemize}
\end{proposition}

\begin{proof}
Let $B\in\calh_2(\frko)$ be a half-integral symmetric matrix associated to $(L, Q)$.
First we consider the case $E/F$ is unramified.
If $f=0$, then we have $B\perp -B\sim H\oplus H$, where $H$ is the hyperbolic plane $\begin{pmatrix} 0 & 1/2 \\ 1/2 & 0 \end{pmatrix}$.
In fact, it is easy to see that $B\perp -B$ expresses $H$, and so $B\perp -B\sim H\perp K$ for some $K\in\calh_2(\frko)$.
Since $-\det (2K)\in \frko^{\times 2}$, we have $K\sim H$.
This settles the case $f=0$ of (1).
Next, we consider the case $0<f$.
Let $\{1, \ome\}$ be a basis for $\frko_E$ as an $\frko$-module.
Then, since $F$ is dyadic, we have $\mathrm{tr}(\ome)\in\frko^\times$.
By multiplying $\ome$ by some unit, we may assume $\mathrm{tr}(\ome)=1$.
By using this basis, the half-integral symmetric matrix associated to $(\frko_E, \caln)$ is of the form $\begin{pmatrix} 1 & 1/2 \\ 1/2 & u\end{pmatrix}$ for some $u\in \frko$.
Since $\{1, \vpi^f\ome\}$ is a basis of $\frko_{E, f}$ over $\frko$, we may assume $B=\begin{pmatrix} 1 & \vpi^f/2 \\ \vpi^f/2 & u \vpi^{2f} \end{pmatrix}$.
If $0<f<2e$, we have
\[
B\perp -B=
\bdm
{\begin{matrix} 1 & \vpi^f/2 \\ \vpi^f/2 & u \vpi^{2f} \end{matrix}}
{\begin{matrix} 0 \phantom{aaaa}& 0 \\ 0\phantom{aaaa} & 0 \end{matrix}}
{\begin{matrix} 0 \phantom{aaa}& 0 \\ 0\phantom{aaa} & 0 \end{matrix}}
{\begin{matrix} -1 & -\vpi^f/2 \\ -\vpi^f/2 & -u \vpi^{2f} \end{matrix}}
\xrightarrow{\left(\begin{smallmatrix} 1 & 0 & 1 & 0 \\  0 & 1 & 0 & 1 \\  0 & 0 & 1 & 0 \\  0 & 0 & 0 & 1 \end{smallmatrix}\right)}
\bdm
{\begin{matrix} 1 & \vpi^f/2 \\ \vpi^f/2 & u \vpi^{2f} \end{matrix}}
{\begin{matrix} 1 & \vpi^f/2 \\ \vpi^f/2 & u \vpi^{2f} \end{matrix}}
{\begin{matrix} 1 & \vpi^f/2 \\ \vpi^f/2 & u \vpi^{2f} \end{matrix}}
{\begin{matrix} 0 \phantom{aaa}& 0 \\ 0\phantom{aaa} & 0 \end{matrix}}.
\]
Here, $X\xrightarrow{A} Y$ means $Y=X[A]$.
Since the last matrix is a reduced form of GK type $((0,f,f,2f),\,  (14)(23))$, we have proved the case $0<f<2e$ of (1).
Next, suppose $f\geq 2e$.
Then we have
\[
B=\begin{pmatrix} 1 & \vpi^f/2 \\ \vpi^f/2 & u \vpi^{2f} \end{pmatrix}
\xrightarrow{\left(\begin{smallmatrix} 1 & -\vpi^f/2 \\ 0 & 1\end{smallmatrix}\right)}
\begin{pmatrix} 1 & 0 \\ 0 & v\vpi^{2f-2e} \end{pmatrix}.
\]
Here, $v=\vpi^{2e}(4u-1)/4\in\frko^\times$.
Then we have
\[
B\perp -B 
\sim
\bdm
{\begin{matrix} 1 & \phantom{-}0 \\ 0 & -1 \end{matrix}}
{\begin{matrix} 0\phantom{aaaaaa} & 0 \\ 0\phantom{aaaaaa} & 0 \end{matrix}}
{\begin{matrix} 0 & \phantom{-}0 \\ 0 & \phantom{-}0 \end{matrix}}
{\begin{matrix} v\vpi^{2f-2e} & 0  \\ 0 & -v\vpi^{2f-2e}  \end{matrix}}
\xrightarrow{\left(\begin{smallmatrix} 1 & 1 & 0 & 0 \\  0 & 1 & 0 & 0 \\  0 & 0 & 1 & 1 \\  0 & 0 & 0 & 1 \end{smallmatrix}\right)}
\bdm
{\begin{matrix} 1 & 1 \\ 1 & 0 \end{matrix}}
{\begin{matrix} 0\phantom{aaaaaa} & 0 \\ 0\phantom{aaaaaa} & 0 \end{matrix}}
{\begin{matrix} 0 & 0 \\ 0 & 0 \end{matrix}}
{\begin{matrix} v\vpi^{2f-2e} & v\vpi^{2f-2e}  \\ v\vpi^{2f-2e}  & 0 \end{matrix}}.
\]
It is easy to check that the lase matrix is a reduced form of GK type $(0, 2e, 2f-2e, 2f),\,  (12)(34))$.
Thus we have proved the last case of (1).

Suppose that $E/F$ is ramified and $d=2g\leq 2e$.
In this case, $E$ is generated by an element $\vpi_E=\vpi^g (-1+\sqrt{\vep})/2$, such that $\ord(\vep-1)=2e-2g+1$.
Then $\{1, \vpi^f\vpi_E\}$ is a basis of $\frko_{E, f}$.
By using this basis, $B=\begin{pmatrix} 1 & \vpi^{g+f}/2 \\ \vpi^{g+f}/2 & u\vpi^{2f+1} \end{pmatrix}$.
Here, $u=\vpi^{2g-1}(1-\vep)/4\in\frko^\times$.
If $f<2e-g$, we have
\begin{align*}
B\perp -B 
\sim&
\bdm
{\begin{matrix} 1 & \vpi^{g+f}/2 \\ \vpi^{g+f}/2 & u \vpi^{2f+1} \end{matrix}}
{\begin{matrix} 0 \phantom{aaaa}& 0 \\ 0\phantom{aaaa} & 0 \end{matrix}}
{\begin{matrix} 0 \phantom{aaa}& 0 \\ 0\phantom{aaa} & 0 \end{matrix}}
{\begin{matrix} -1 & -\vpi^{g+f}/2 \\ -\vpi^{g+f}/2 & -u \vpi^{2f+1} \end{matrix}}
\\
\xrightarrow{\left(\begin{smallmatrix} 1 & 0 & 1 & 0 \\  0 & 1 & 0 & 1 \\  0 & 0 & 1 & 0 \\  0 & 0 & 0 & 1 \end{smallmatrix}\right)}
&
\bdm
{\begin{matrix} 1 & \vpi^{g+f}/2 \\ \vpi^{g+f}/2 & u \vpi^{2f+1} \end{matrix}}
{\begin{matrix} 1 & \vpi^{g+f}/2 \\ \vpi^{g+f}/2 & u \vpi^{2f+1} \end{matrix}}
{\begin{matrix} 1 & \vpi^{g+f}/2 \\ \vpi^{g+f}/2 & u \vpi^{2f+1} \end{matrix}}
{\begin{matrix} 0 \phantom{aaaaa}& 0 \\ 0\phantom{aaaaa} & 0 \end{matrix}}
\\
\xrightarrow{\left(\begin{smallmatrix} 1 & 0 & 0 & 0 \\  0 & 1 & 0 & 0 \\  0 & 0 & 1 & 0 \\  0 & -1 & 0 & 1 \end{smallmatrix}\right)}
&
\bdm
{\begin{matrix} \phantom{aaa}1 & \phantom{a}0 \\ \phantom{aaa}0 & \phantom{aa}-u \vpi^{2f+1} \end{matrix}}
{\begin{matrix} 1 & \vpi^{g+f}/2 \\ \vpi^{g+f}/2 & u \vpi^{2f+1} \end{matrix}}
{\begin{matrix} 1\phantom{a} & \vpi^{g+f}/2 \\ \vpi^{g+f}/2 & u \vpi^{2f+1} \end{matrix}}
{\begin{matrix} 0 \phantom{aaaaa}& 0 \\ 0\phantom{aaaaa} & 0 \end{matrix}}.
\end{align*}
The last matrix is a reduced form of GK type 
$((0, 2f+1, 2g-1, 2g+2f),\, (14)(23))$ if $0\leq f \leq g-1$,
and a reduced form of GK type  $((0, g+f, g+f, 2g+2f),\, (14)(23))$, if 
$g\leq f < 2e-g$.
This proves the first and the second case of (2).
Suppose that $f\geq 2e-g$.
Then we have
\[
B=\begin{pmatrix} 1 & \vpi^{g+f}/2 \\ \vpi^{g+f}/2 & u\vpi^{2f+1} \end{pmatrix}
\xrightarrow{\left(\begin{smallmatrix} 1 & -\vpi^{g+f}/2 \\ 0 & 1\end{smallmatrix}\right)}
\begin{pmatrix} 1 & 0 \\ 0 & v \vpi^{2g+2f-2e} \end{pmatrix}.
\]
Here, $v=-\vpi^{2e}\vep/4\in\frko^\times$.
In this case, by a similar calculation as before, we have
\[
B\perp -B 
\sim
\bdm
{\begin{matrix} 1 & 1 \\ 1 & 0 \end{matrix}}
{\begin{matrix} 0\phantom{aaaamaaa} & 0 \\ 0\phantom{aaaamaaa} & 0 \end{matrix}}
{\begin{matrix} 0 & 0 \\ 0 & 0 \end{matrix}}
{\begin{matrix} v\vpi^{2g+2f-2e} & v\vpi^{2g+2f-2e}  \\ v\vpi^{2g+2f-2e}  & 0 \end{matrix}}.
\]
This matrix is a reduced form of GK type $((0, 2e, 2g+2f-2e, 2g+2f),\, (12)(34))$, and this settles the last case of (2).

Finally, suppose that $E/F$ is ramified and $d=2e+1$.
In this case, the quadratic extension $E/F$ is generated by $\vpi_E=\sqrt{-\vpi u}$ for some unit $u\in\frko^\times$.
Then $\{1, \vpi^f\vpi_E\}$ is a $\frko$-basis of $\frko_{E, f}$. 
By using this basis,  we may assume 
$B=\begin{pmatrix} 1 & 0 \\ 0 & u \vpi^{2f+1} \end{pmatrix}$.
Then, by a similar calculation as before, we have
\[
B\perp -B 
\sim
\begin{cases}
\bdm
{\begin{matrix} 1 & 0 \\ 0 & u \vpi^{2f+1} \end{matrix}}
{\begin{matrix} 1 & 0 \\ 0 & u \vpi^{2f+1} \end{matrix}}
{\begin{matrix} 1 & 0 \\ 0 & u \vpi^{2f+1} \end{matrix}}
{\begin{matrix} 0 \phantom{am} & 0\phantom{am} \\ 0\phantom{am} & 0\phantom{am} \end{matrix}}
& \text{ if $f<e$, } \\
\noalign{\vskip 5pt} 
\bdm
{\begin{matrix} 1 & 1 \\ 1 & 0 \end{matrix}}
{\begin{matrix} 0\phantom{aamaa} & 0 \\ 0\phantom{aamaa} & 0 \end{matrix}}
{\begin{matrix} 0 & 0 \\ 0 & 0 \end{matrix}}
{\begin{matrix} u\vpi^{2f+1} & u\vpi^{2f+1}  \\ u\vpi^{2f+1}  & 0 \end{matrix}}
& \text{ if $f\geq e$. } 
\end{cases}
\]
If $0\leq f < e$, then the first matrix is a reduced form of GK type 
$((0, 2f+1, 2e, 2e+2f+1),\, (13)(24))$.
If $ f \geq e$, then the second matrix is a reduced form of GK type 
$((0, 2e, 2f+1, 2e+2f+1),\, (12)(34))$.
Hence we have proved the proposition.
\end{proof}

By Theorem \ref{thm5.1} (Corollary 5.1 of \cite{IK1}), we obtain the following proposition.
\begin{proposition} 
\label{prop:A.8}
Suppose that $(L, Q)=(\frko_{E,f}, \caln)$.
\begin{itemize}
\item[(1)] 
Assume that $E$ is unramified.
Then we have
\[
\GK(L\oplus -L)=
\begin{cases}
(0,0,0,0) & \text{ if $f=0$,} 
\\ \noalign{\vskip 3pt}
(0,f,f,2f) & \text{ if $0 < f < 2e$,} 
\\ \noalign{\vskip 2pt}
(0, 2e, 2f-2e, 2f) & \text{ if $f\geq 2e$.}
\end{cases}
\]
\item[(2)] 
Assume that $E$ is ramified and that $d=2g\leq 2e$.
Then we have
\[
\GK(L\oplus -L)=
\begin{cases}
(0, 2f+1, 2g-1, 2g+2f) & \text{ if $0\leq f \leq g-1$,} 
\\ \noalign{\vskip 5pt}
(0, g+f, g+f, 2g+2f) & \text{ if $g\leq f < 2e-g$,} 
\\ \noalign{\vskip 5pt}
(0, 2e, 2g+2f-2e, 2g+2f) & \text{ if $ f \geq 2e-g$.} 
\end{cases}
\]
\item[(3)] 
Assume that $E$ is ramified and that $d=2e+1$.
Then we have
\[
\GK(L\oplus -L)=
\begin{cases}
(0, 2f+1, 2e, 2e+2f+1) & \text{ if $0\leq f < e$,} 
\\ \noalign{\vskip 5pt}
(0, 2e, 2f+1, 2e+2f+1) & \text{ if $ f \geq e$.} 
\end{cases}
\]

\end{itemize}
\end{proposition}

We shall give a Jordan splitting for $L=(\frko_{E,f}, \caln)$.
Let $L=\bigoplus L_i$ be a Jordan splitting such that $L_i$ is $i$-modular. 
Put $\mathrm{Jor}(L)=\{i\in\ZZ\,|\, L_i \text{ is nonzero.}\}$.
\begin{lemma}
\label{lem:A10}
\begin{itemize}
\item[(1)] 
Suppose that $E/F$ is unramified.
If $f<e$, then $\mathrm{Jor}(L)=\{f-e\}$ and $L$ is an indecomposable $(f-e)$-modular lattice.
If $f\geq e$, then $\mathrm{Jor}(L)=\{0, 2f-2e\}$ and $L\sim (1)\perp (u\vpi^{2f-2e})$, with $u\in\frko^\times$.
\item[(2)]
Suppose that $E/F$ is ramified and $d=2g\leq 2e$.
If $f<e-g$, then $\mathrm{Jor}(L)=\{f+g-e\}$ and $L$ is an indecomposable $(f+g-e)$-modular lattice.
If $f\geq e-g$, then  $\mathrm{Jor}(L)=\{0, 2f+2g-2e\}$ and $L\sim (1)\perp (u\vpi^{2f+2g-2e})$, with $u\in\frko^\times$.
\item[(3)]
Suppose that $E/F$ is ramified and $d=2e+1$.
In this case, $\mathrm{Jor}(L)=\{0, 2f+1\}$ and 
$L\sim (1)\perp (u\vpi^{2f+1})$, with $u\in\frko^\times$.
\end{itemize}
\end{lemma}
\begin{proof}
Suppose that $E/F$ is unramified.
As we have seen in the proof of Proposition \ref{prop:A7}, $L$ is expressed by 
$B=\begin{pmatrix} 1 & \vpi^f/2 \\ \vpi^f/2 & u \vpi^{2f} \end{pmatrix}$ for some $u\in\frko^\times$.
If $f<e$, then $B$ is indecomposable by Lemma 2.1 of \cite{IK1}.
In this case, it is easy to see $\vpi^{e-f}B$ is unimodular.
If $f\geq e$, then we have $B\sim (1) \perp ((-1+u\vpi^{2e}) \vpi^{2f-2e})$.
This proves (1).
The other cases can be proved similarly.
\end{proof}

We shall calculate the Gross-Keating invariant $\GK(L\cap\vpi^i L^\sharp)$ for $(L\cap\vpi^i L^\sharp, \vpi^{-i}Q)$ for each $i\in \mathrm{Jor}(L)$.
Recall that the Gross-Keating invariant $(a_1, a_2)$ of a binary form $(L', Q')$ is determined by
\begin{align*}
a_1&=\mathrm{ord}(\bfn(L')), \\
a_1+a_2&=
\begin{cases}
\ord(4\det Q') & \text{ if $\ord(\frkD_{Q'})=0$,}
\\
\ord(4\det Q')-\ord(\frkD_{Q'})+1 & \text{ if $\ord(\frkD_{Q'})>0$.}
\end{cases}
\end{align*}
Here, $\frkD_{Q'}$ is the discriminant of $F(\sqrt{-\det Q'})/F$.
These formula follows form Proposition \ref{prop:6.2}, since a binary quadratic form is isomorphic to some $(\frko_{E,f}, \caln)$ up to multiplication by a unit (\cite{IK1}, Proposition 2.1).
In terms of $B=\begin{pmatrix} b_{11} & b_{12} \\ b_{12} & b_{22}\end{pmatrix}\in \calh_2(\frko)$, the Gross-Keating invariant $(a_1, a_2)$ of $B$ is given by 
\begin{align*}
a_1&=\min\{\mathrm{ord}(b_{11}), \mathrm{ord}(2b_{12}), \mathrm{ord}(b_{22})\}, \\
a_1+a_2&=
\begin{cases}
\ord(4\det B) & \text{ if $\ord(\frkD_{B})=0$,}
\\
\ord(4\det B)-\ord(\frkD_{B})+1 & \text{ if $\ord(\frkD_{B})>0$.}
\end{cases}
\end{align*}
Note also that $\GK(\vpi^i B)=(a_1+i, a_2+i)$.

\begin{proposition} 
\label{prop:A.9}
Suppose that $(L, Q)=(\frko_{E,f}, \caln)$ and $i\in \mathrm{Jor}(L)$.
\begin{itemize}
\item[(1)] 
Assume that $E$ is unramified.
Then we have
\[
\GK(L\cap \vpi^i L^\sharp)=
\begin{cases}
(e-f, e+f) & \text{ if $f< e$,}
\\
(0, 2f) & \text{ if $f\geq e$.}
\end{cases}
\]
\item[(2)] 
Assume that $E$ is ramified and that $d=2g\leq 2e$.
Then, we have
\[
\GK(L\cap \vpi^i L^\sharp)=
\begin{cases}
(e-g-f, e-g+f+1)
 & \text{ if $f<e-g$,}
\\
(0, 2f+1) & \text{ if $f\geq e-g$,}
\end{cases}
\]
\item[(3)] 
Assume that $E$ is ramified and that $d=2e+1$.
Then we have
\[
\GK(L\cap \vpi^i L^\sharp)=
(0, 2f+1).
\]
\end{itemize}
\end{proposition}

\begin{proof}
Suppose that $L$ is $i$-modular.
In this case, $L\cap \vpi^i L^\sharp=L$.
Then $\GL(L\cap \vpi^i L^\sharp)=(a_1-i, a_2-i)$, where $(a_1, a_2)=\GL(L)$.
(Remember that the quadratic form for $L\cap\vpi^i L^\sharp$ is multiplied by $\vpi^{-i}$.)

Suppose that $L\sim (1)\perp (u\vpi^k)$.
In this case, $\mathrm{Jor}(L)=\{0, k\}$ and $(L\cap \vpi^i L^\sharp, \vpi^{-i} Q)$ is expressed by $(1)\perp (u\vpi^k)$ or $(u)\perp (\vpi^k)$, according as $i=0$ or $i=k$.
In either case, $(L\cap \vpi^i L^\sharp, \vpi^{-i} Q)$ is weakly equivalent to $(L, Q)$.
Hence the proposition.
\end{proof}

For $B\in \calh_n(\frko)$, we define $\EGK(B)^{\leq 1}$ as in the main part of this paper.
This is defined as follows.
Let $\GK(B)=(\underbrace{0, \ldots, 0}_{m_0}, \underbrace{1, \ldots, 1}_{m_1}, a_{m_0+m_1+1}, \ldots, a_n)$, where $a_{m_0+m_1+1}>1$.
If $B$ is equivalent to a reduced form
\[
\begin{array}{ccccc}
&\hskip -32pt
\overbrace{\hphantom{B_{11}}}^{m_0} \;
\overbrace{\hphantom{ B_{12}}}^{m_1}  \;
\overbrace{\hphantom{ B_{15}}}^{n-m_0-m_1} \\ & 
B'=\left(
\begin{array}{cccl}
B_{00} & B_{01} & B_{02}  \\
{}^t\! B_{01} & B_{11} & B_{12}  \\
{}^t\! B_{02} & {}^t\! B_{12} & B_{22}  \\
\end{array} \hskip -0pt
\right) \hskip -5pt
\begin{array}{l}
 \left.\vphantom{B_{11}} \right\} \text{\footnotesize${m_0}$} \\
 \left.\vphantom{B_{11}} \right\} \text{\footnotesize${m_1}$} \\
 \left.\vphantom{B_{11}} \right\} \text{\footnotesize${n-m_0-m_1}$,} 
\end{array}
\end{array}
\]
then $\EGK(B)^{\leq 1}=\EGK\left(\begin{pmatrix} B_{00} & B_{11} \\ {}^t\! B_{01} & B_{11}\end{pmatrix}\right)$.
This definition does not depend on the choice of the reduced form $B'$.
If $B$ is associated to a quadratic lattice $M$, we write $\EGK(M)^{\leq 1}$ for $\EGK(B)^{\leq 1}$.

The next proposition follows from Proposition \ref{prop:A.9}.
\begin{proposition}
\label{prop:A.11}
Suppose that $(L, Q)=(\frko_{E,f}, \caln)$ and $i\in \mathrm{Jor}(L)$.
\begin{itemize}
\item[(1)] 
Assume that $E$ is unramified.
Then we have 
\[
\EGK(L\cap \vpi^i L^\sharp)^{\leq 1}=
\begin{cases}
\emptyset & \text{ if $f<e-1$,}
\\
(2;1;\xi) & \text{ if $f=0$, $e=1$, }
\\
(1;1;1) & \text{ if $f=e-1$, $e>1$, }
\\
(1;0;1) & \text{ if $f\geq e$.}
\end{cases}
\]
\item[(2)] 
Assume that $E$ is ramified and that $d=2g\leq 2e$.
Then, we have
\[
\EGK(L\cap \vpi^i L^\sharp)^{\leq 1}=
\begin{cases}
\emptyset & \text{ if $f<e-g-1$,}
\\
(1;1;1)
 & \text{ if $f=e-g-1$,}
\\
(1;0;1) & \text{ if $f\geq e-g$, \, $g<e$,}
\\
(1;0;1) & \text{ if $f>0$, \, $g=e$,}
\\
(1,1;0,1;1,0) & \text{ if $f=0$, \, $g=e$.}
\end{cases}
\]
\item[(3)] 
Assume that $E$ is ramified and that $d=2e+1$.
Then we have
\[
\EGK(L\cap \vpi^i L^\sharp)^{\leq 1}=
\begin{cases}
(1,1;0,1;1,0) & \text{ if $f=0$,}
\\
(1;0,1) & \text{ if $f>0$.}
\end{cases}
\]
\end{itemize}
\end{proposition}

We shall show that there exist two binary quadratic lattices $L$ and $L'$, which satisfy the following conditions (1), (2), and (3).
\begin{itemize}
\item[(1)] $\GK(L\perp -L)=\GK(L'\perp -L')$.
\item[(2)] $\mathrm{Jor}(L)=\mathrm{Jor}(L')$ and $\EGK(L\cap \vpi^i L^\sharp)^{\leq 1}=\EGK(L'\cap \vpi^i L'^\sharp)^{\leq 1}$ for each $i\in \mathrm{Jor}(L)$.
\item[(3)] $\beta(L)\neq \beta(L')$.
\end{itemize}
\begin{example}
\label{ex:A1}
Suppose that $e=5$.
Suppose also that $E/F$ is a ramified quadratic extension with $d=2$ and $E'/F$ is a ramified quadratic extension with $d=4$.
Put $L=\frko_{E, 2}$ and $L'=\frko_{E', 1}$.
Then we have 
\[
\GK(L\perp -L)=\GK(L'\perp -L')=(0,3,3,6)
\]
by Proposition \ref{prop:A.8}.
Note that $\mathrm{Jor}(L)=\mathrm{Jor}(L')=\{-2
\}$ 
and 
\[
\EGK(L\cap \vpi^{-1} L^\sharp)^{\leq 1}=\EGK(L'\cap \vpi^{-1} L'^\sharp)^{\leq 1}=\emptyset
\qquad (i\in \mathrm{Jor}(L))
\]
by Proposition \ref{prop:A.11}.
But we have
\[
\bet(L)=q^7, 
\qquad 
\bet(L')=2q^7
\]
by Proposition \ref{prop:A6}.
Thus $\GL(L\perp -L)$ and $\EGK(L\cap \vpi^i L^\sharp)^{\leq 1}$ are not enough to determine $\bet(L)$ in the case $e>1$.
\end{example}


\end{document}